\newcommand{\C}{\mathbb{C}}

\newcommand{\T}{\mathbb{T}}

\documentclass[pdflatex,fleqn]{article}

\def\firstpage{1}

\usepackage{}
\usepackage{xcolor}
\usepackage{mathrsfs}
\usepackage{amsfonts}
\usepackage{amssymb}
\usepackage{latexsym}
\usepackage{amsmath}
\usepackage{amsmath,amssymb,amsfonts}
\usepackage{mathrsfs}
\usepackage{amscd}
\usepackage{graphicx}
\usepackage{hyperref}
\usepackage{hypernat}
\usepackage{cite}
\usepackage{mathtools}
\usepackage{geometry,amsthm,graphics,amssymb,amsmath,enumerate,latexsym,tabularx,shapepar}
\usepackage[all,2cell,dvips]{xy}\UseAllTwocells\SilentMatrices

\usepackage{extarrows}
\date{}

\theoremstyle{plain}
\newtheorem{theorem}{Theorem}[section]

\newtheorem{corollary}[theorem]{Corollary}
\newtheorem{lemma}[theorem]{Lemma}

\newtheorem{definition}[theorem]{Definition}
\theoremstyle{remark}
\newtheorem{remark}[theorem]{Remark}

\newtheorem{numb}[theorem]{}

\title{\bf ON THE DECOMPOSITION THEOREMS FOR $C^*$-ALGEBRAS}

\author{Chunlan Jiang$^1$  Liangqing Li$^2$, and Kun Wang$^{3*}$}

\begin{document}

\large

\maketitle

\renewcommand{\thefootnote}{\fnsymbol{footnote}}

\footnotetext{\hspace*{-5mm} \begin{tabular}{@{}r@{}p{13.4cm}@{}}
$^1$ & College of Mathematics and Information Science, Hebei Normal University,  Shijiazhuang,  050024,    China.\\
&{E-mail: cljiang@hebtu.edu.cn} \\
$^{2}$ & Department of Mathematics, University of Puerto Rico at Rio Piedras, PR 00936, USA.\\
&{E-mail: liangqing.li@upr.edu}\\
$^{3}$ & Department of Mathematics, University of Puerto Rico at Rio Piedras, PR 00936, USA.\\
&{E-mail: kun.wang@upr.edu}\\
$^*$ & Corresponding author.
\end{tabular}}

\renewcommand{\thefootnote}{\arabic{footnote}}



{\bf Abstract}
Elliott dimension drop interval algebra is an  important class among all $C^*$-algebras in the classification theory.
Especially, they are building stones of $\mathcal{AHD}$ algebra and the latter contains all 
$AH$ algebras with the ideal property of no dimension growth.
In this paper, we will show two decomposition theorems related to the Elliott dimension drop interval algebra.
Our results are  key steps in classifying all $AH$ algebras with the ideal property of no dimension growth. \\

{\bf{Keywords}:} $C^*$-algebra, Elliott dimension drop interval algebra,  decomposition theorem, spectral distribution property

{\bf AMS subject classification}: Primary:  46L35, 46L80.
 

\section{Introduction}

Classification theorems have been obtained for $AH$ algebras---the inductive limits of cut downs of matrix algebras over compact metric spaces by projections---and $AD$ algebras---the inductive limits of Elliott dimension drop interval algebras in two special cases:

 $1$. Real rank zero case: all such $AH$ algebras
with no dimension growth and such $AD$ algebras (See \cite{Ell1}, \cite{EGLP}, \cite{EG1}, \cite{EG2}, \cite{EGS}, \cite{D}, \cite{G1}-\cite{G4}, \cite{Ei}, and \cite{DG});

$2$. Simple case: all
such $AH$ algebras with no dimension growth (which includes all simple $AD$ algebras by \cite{EGJS}) (See \cite{Ell2}, \cite{Ell3}, \cite{NT}, \cite{Thm2}, \cite{Thm3}, \cite{Li1}-\cite{Li4}, \cite{G5}, and
\cite{EGL1}).

In \cite{EGL1}, the authors pointed out two important possible next steps after the completion of  classification of simple $AH$ algebras (with no dimension
growth). One of these is the classification of simple $ASH$ algebras---the simple inductive limits of subhomogeneous algebras (with no dimension growth). The
other is to generalize and  unify the above-mentioned classification theorems for simple $AH$ algebras and real rank zero $AH$ algebras by classifying
$AH$ algebras with the ideal property. In this article, we have achieved several  key results for the second goal  by providing two decomposition theorems.

As in \cite{EG2}, let $T_{\uppercase\expandafter{\romannumeral2},k}$ be the $2$-dimensional connected simplicial complex with $H^{1}(T_{\uppercase\expandafter{\romannumeral2},k})=0$ and $H^{2}(T_{\uppercase\expandafter{\romannumeral2},k})=\mathbb{Z}/k\mathbb{Z}$, and
let $I_{k}$ be the subalgebra of $M_{k}(C[0,1])$ defined by 
$$I_k=\{  f\in M_k(C[0,1]):~f(0) \in\mathbb{C} \cdot 1_k\mbox{ and }f(1) \in\mathbb{C} \cdot 1_k\}.$$
This algebra is called an Elliott dimension drop interval algebra. Denote by $\mathcal{HD}$ the class of algebras consisting of direct sums of building
blocks of the forms $M_{l}(I_{k})$ and $PM_{n}(C(X))P$, with $X$ being one of the spaces $\{pt\}$, $[0,1]$, $S^{1}$, and $T_{\uppercase\expandafter{\romannumeral2},k}$, and with $P\in M _{n}(C(X))$ being a
projection. (In \cite{DG}, this class is denoted by $SH(2)$, and in \cite{Jiang1}, this class is denoted by $\mathcal{B}$). We will call a $C^{*}$-algebra an $A\mathcal{HD}$
algebra, if it is an inductive limit of algebras in $\mathcal{HD}$. 
In \cite{GJLP1}, \cite{GJLP2}, \cite{Li4},  and \cite{Jiang2}, it is proved that
all $AH$ algebras with the ideal property of no dimension growth are inductive limits of algebras in the class $\mathcal{HD}$---that is, they are $A\mathcal{HD}$ algebras. 
By this reduction theorem, to classify $AH$ algebras with the ideal property, we must study 
the properties of homomorphisms between those basic building blocks.


In the local uniqueness theorem for classification, it requires the homomorphisms involved to satisfy a certain spectral distribution property, called the  $sdp$
property (more specifically, $sdp(\eta,\delta)$ property  introduced in \cite{G5} and \cite{EGL1} for some positive real numbers $\eta$ and $\delta$). 
This property automatically
 holds for the homomorphisms $\phi_{n,m}$ (provided that $m$ is large enough) giving rise to  a simple inductive limit procedure. 
 But for the case of general inductive limit $C^{*}$-algebras
with the ideal property,  to obtain this $sdp$ property, we must pass to certain good quotient algebras which   corresponding to
simplicial sub-complexes of the  original spaces; a uniform uniqueness theorem, that does not depend on the choice of simplicial sub-complexes involved, is
required. 
For the case of an interval, whose simplicial sub-complexes are finite unions of subintervals and points, such a uniform uniqueness theorem is proved
in \cite{Ji-Jiang} (see \cite{Li2} and \cite{Ell2} also). 
But for the general case, there are no uniqueness theorem for the general case
involving arbitrary finite subsets of $M_{n}(C(T_{\uppercase\expandafter{\romannumeral2},k}))$ (or $M_{l}(I_{k})$). 
In this paper, we prove decomposition theorems between such building blocks or between
a building block of this kind and a  homogeneous building block.
And we will compare the decompositions of two different homomorphisms in the last part of chapter 4.
Such decomposition and comparison results will be used in the proof of the uniqueness
theorem for AH algebras with the ideal property in \cite{GJL} by  Gong, Jiang and Li.


\vspace{3mm}

\section{Notation and terminology}

\vspace{-2.4mm}

In this section, we will introduce some notation and terminology.

\begin{definition} \label{1.14} Let $X$ be a compact metric space and $\psi: C(X)\rightarrow PM_{k_{1}}(C(Y))P$ (with $rank(P)=k$) be a unital homomorphism. For any point $y\in Y$, there are $k$ mutually orthogonal rank $1$  projections $p_{1},p_{2},\cdots,p_{k}$ with $\sum\limits_{i=1}\limits^{k}p_{i}=P(y)$ and $\{x_{1}(y),x_{2}(y),\cdots,x_{k}(y)\}\subset X$ (may be repeat) such that$$\psi(f)(y)=\sum\limits_{i=1}\limits^{k}f(x_{i}(y))p_{i}, \forall f\in C(X).$$ 
We  denote the set $\{x_{1}(y),x_{2}(y),\cdots,x_{k}(y)\}$ (counting multiplicities), by $Sp\psi_{y}$. We shall call $Sp\psi_{y}$ {\bf the spectrum of } $\psi$
{\bf at the point} $y$.
\end{definition}

\begin{numb}\label{1.15}
For any $f\in I_{k}\subset M_{k}(C[0,1])=C([0,1],M_{k}(\mathbb{C}))$ as in 3.2 of \cite{EGS}, let function $\underline{f}: [0,1]\longrightarrow \mathbb{C}\sqcup M_{k}(\mathbb{C})$ (disjoint union) be defined by
$$\underline{f}(t)=\begin{dcases}
 \lambda, & if\; t=0\; and\; f(0)=\lambda\textbf{1}_{k}\\
  \mu, & if\; t=1\; and \;f(1)=\mu\textbf{1}_{k}\\
  f(t), & if\; 0<t<1~~~~~~~~~~~~~~~~.
   \end{dcases}$$
That is, $\underline{f}(t)$ is the value of irreducible representation of $f$ corresponding to the point  $t$. Similarly,  for $f\in M_{l}(I_{k})$, we can
define $\underline{f}: [0,1]\longrightarrow M_{l}(\mathbb{C})\sqcup M_{lk}(\mathbb{C})$, by
$$\underline{f}(t)= \begin{cases}
  a, & if\;t=0\;and\;f(0)=a\otimes\textbf{1}_{k} \\
   b, & if\;t=1\;and\;f(1)=b\otimes\textbf{1}_{k} \\
    f(t),& if\;0<t<1~~~~~~~~~~~~~~~~~~.
   \end{cases}$$
\end{numb}

\begin{numb}\label{1.16}
Suppose that $\phi: I_{k}\longrightarrow PM_{n}(C(Y))P$ is  a unital homomorphism. Let $r=rank(P)$. For each $y\in Y$, there are $t_{1},t_{2},\cdots,t_{m}\in[0,1]$ and a unitary
$u\in M_{n}(\mathbb{C})$ such that
$$P(y)=u\left(
          \begin{array}{cc}
            \textbf{1}_{rank(P)} & 0 \\
            0 & 0 \\
          \end{array}
        \right)u^{*}$$
and
\begin{equation}\label{dhom}
 \phi(f)(y)=u\left(
                               \begin{array}{ccccc}
                                 \underline{f}(t_{1}) &  & & & \\
                                  & \underline{f}(t_{2}) &  &  &  \\
                                  &  & \ddots &  &  \\
                                  &  &  & \underline{f}(t_{m}) &  \\
                                  &  &  &  &  {\bf 0}_{n-r} \\
                               \end{array}
                             \right)
u^{*}\in P(y)M_{n}(\mathbb{C})P(y)
\end{equation}
for all $f\in I_{k}.$

\end{numb}

\begin{numb}\label{1.17}
Let $\phi$ be the homomorphism  defined by  the equation (2.1) above  with $t_{1},t_{2},$ $\cdots,t_{m}$ as appeared in the diagonal of the matrix. We define the set $Sp\phi_{y}$ to be the points $t_{1},t_{2},\cdots,t_{m}$ with possible
fraction multiplicity. If $t_{i}=0$ or $1$, we will assume that the multiplicity of $t_{i}$ is  $\frac{1}{k}$; if $0<t_{i}<1$, we will assume that the multiplicity of  $t_{i}$ is $1$. For example if we assume $$t_{1}=t_{2}=t_{3}=0<t_{4}\leq t_{5}\leq \cdots \leq t_{m-2}<1=t_{m-1}=t_{m},$$ then
$Sp\phi_{y}=\{0^{\sim\frac{1}{k}},0^{\sim\frac{1}{k}},0^{\sim\frac{1}{k}},t_{4},t_{5},\cdots,t_{m-2},1^{\sim\frac{1}{k}},1^{\sim\frac{1}{k}}\}$, which  can also be written
as $$Sp\phi_{y}=\{0^{\sim\frac{3}{k}},t_{4},t_{5},\\ \cdots,t_{m-2},1^{\sim\frac{2}{k}}\}.$$ Here we emphasize that, for $t\in (0,1)$, we do not allow the multiplicity of $t$ to be non-integral. Also for 0 or 1, the multiplicity must be multiple of $\frac{1}{k}$ (other fraction numbers are not allowed).

Let $\psi: C[0,1]\longrightarrow PM_{n}(C(Y))P$ be defined by the following composition $$\psi:~C[0,1]\hookrightarrow I_{k}\xrightarrow {\phi}PM_{n}(C(X))P,$$ where
the first map is the canonical inclusion. Then we have $Sp\psi_{y}=\{Sp\phi_{y}\}^{\sim k}$---that is, for each element $t\in(0,1)$, its multiplicity in $Sp\psi_{y}$ is exactly $k$ times of the multiplicity in $\phi_{y}$.
\end{numb}

\begin{numb}
\begin{enumerate}[(a)]
\item we use $\sharp(.)$ to denote the cardinal number of a set. Very often, the sets under consideration will be sets with multiplicity, in which case we shall
also count multiplicity when we use the notation $\sharp$.
The set may also contain fractional point. For example, 
$$\sharp\{0_1,1_2, 0,0,1\}=5.$$
\item We shall use $a^{\sim k}$ to denote $\underbrace{a,a,\cdots a}\limits_{k}$. For example $\{a^{\sim3},b^{\sim2}\}=\{a,a,a,b,b\}$.\\
\item For any metric space $X$, any $x_{0}\in X$ and $c>0$, let $B_{c}(x_{0})\triangleq\{x\in X|d(x,x_{0})<c\}$,  the open ball with radius $c$ and center $x_{0}$.\\
\item Suppose that $A$ is a $C^{*}$-algebra, $B\subset A$ a subset (often a subalgebra), $F\subset A$ is a finite subset and $\varepsilon>0$. If for each element $f\in F$, there is an element $g\in B$ such
that $\|f-g\|<\varepsilon$, then we shall say that $F$ is approximately contained in $B$ to within $\varepsilon$, and denote this by $F\subset_{\varepsilon}B$.\\
\item Let $X$ be a compact metric space. For any $\delta>0$, a finite set $\{x_{1},x_{2},\cdots,x_{n}\}$ is said to be $\delta$-dense in $X$ if for
any $x\in X$, there is $x_{i}\in \{x_{1},x_{2},...,x_{n}\}$ such that $dist(x,x_{i})<\delta$.\\
\item We shall use $\bullet$ or $\bullet\bullet$ to denote any possible positive integers.\\
\item For any two projections $p,q\in A$, by $[p]\leq[q]$ we mean that $p$ is unitarily equivalent to a sub-projection of $q$. And we
use $p\sim q$ to denote that $p$ is unitarily equivalent to $q$.

\end{enumerate}
\end{numb}

\begin{numb}\label{1.18}
Let $A=M_{l}(I_{k})$. Then every  point $t\in(0,1)$ corresponds to an irreducible representation $\pi_{t}$,  defined by $\pi_{t}(f)=f(t)$. The representations  $\pi_{0}$ and $\pi_{1}$ defined by $$\pi_{0}=f(0)~~~~~~\mbox{ and}~~~~~~~\pi_{1}=f(1)$$ are no longer irreducible. We use $\underline{0}$ and $\underline{1}$ to denote the corresponding points for the irreducible representations. That is, $$\pi_{\underline{0}}(f)=\underline{f}(0),~~~~~~~\mbox {and}~~~~~~~~\pi_{\underline{1}}(f)=\underline{f}(1).$$ Or we can also write $\underline{f}(0)\triangleq f(\underline{0})$ and $\underline{f}(1)\triangleq f(\underline{1})$. Then the equation $(*)$ could be written as
$$\phi(f)(y)=u\left(
                \begin{array}{ccccc}
                  f(t_{1}) &  &  &  &  \\
                   & f(t_{2}) &  &  &  \\
                  &  & \ddots &  &  \\
                   & &  & f(t_{m}) &  \\
                   &  &  &  & {\bf 0}_{n-r} \\
                \end{array}
              \right)u^{*},$$
where some of $t_{i}$ may be $\underline{0}$ or $\underline{1}$. In this notation,  up to unitary equivalence, 
$f(0)$ is  equal to  diag$(\underbrace{f(\underline{0}),f(\underline{0}),\cdots,f(\underline{0})}\limits_{k})$ .


Under this notation, we can also write $0^{\sim\frac{1}{k}}$ as $\underline{0}$. Then the example of $Sp\phi_{y}$ in \ref{1.17} can be  written as
$$Sp\phi_{y}=\{0^{\sim\frac{1}{k}},0^{\sim\frac{1}{k}},0^{\sim\frac{1}{k}},t_{4},t_{5},\cdots,t_{m-2},1^{\sim\frac{1}{k}},1^{\sim\frac{1}{k}}\}\\=\{\underline{0},\underline{0},\underline{0},
t_{4},t_{5},\cdots,t_{m-2},\underline{1},\underline{1}\}.$$
\end{numb}

\begin{numb}\label{1.19}
For a homomorphism $\phi: A\longrightarrow M_{n}(I_{k})$, where $A=I_{k}$ or $C(X)$, and for any $t\in[0,1]$, define $Sp\phi_{t}=Sp\psi_{t}$, where
$\psi$ is defined by the composition $$\psi: A\xrightarrow {\phi} M_{n}(I_{l})\rightarrow M_{nl}(C[0,1]).$$
Also $Sp\phi_{\underline{0}}=Sp(\pi_{\underline{0}}\circ \phi)$. Hence, $Sp\phi_{0}=\{Sp\phi_{\underline{0}}\}^{\sim k}.$
\end{numb}

\begin{numb}\label{1.20}
Let $\phi: M_{n}(A)\longrightarrow B$ be a unital homomorphism. It is well known (see 1.34 and 2.6 of \cite{EG2}) that there is an identification of $B$
with $(\phi(e_{11})B\phi(e_{11}))\otimes M_{n}(\mathbb{C})$ such that $$\phi=\phi_{1}\otimes id_{n}: M_{n}(A)=A\otimes M_{n}(\mathbb{C})\longrightarrow (\phi(e_{11})B\phi(e_{11}))\otimes M_{n}(\mathbb{C})=B,$$ where $e_{11}$ is the matrix unit of upper left corner of $M_{n}(A)$ and $\phi_{1}=\phi|_{e_{11}M_{n}(A)e_{11}}: A\longrightarrow \phi(e_{11})B\phi(e_{11})$.

If we further assume that $A=I_{k}$ or $C(X)$ (with $X$ being a connected $CW$ complex) and $B$ is either $QM_{n}(C(Y))Q$ or $M_{l}(I_{k_{1}}),$ then for any
$y\in SpB$, define $Sp\phi_{y}\triangleq Sp(\phi_{1})_{y}.$
Here,  we use the standard notation that if $B=PM_{m}(C(Y))P$ then $SpB=Y$; and if $B=M_{l}(I_{k})$, then $Sp(B)=[0,1].$
\end{numb}

\begin{numb}\label{1.24} Let $A$ and $B$ be either of form $PM_{n}(C(X))P$ (with $X$ path connected) or of form $M_{l}(I_{k})$. Let $\phi: A\longrightarrow B$ be a unital
homomorphism, we say that $\phi$ has property $sdp(\eta,\delta)$ ({\bf spectral distribution property with respect to $\eta$ and  $\delta$}) if for any $\eta$-ball $$B_{\eta}(x)=\{x^{\prime}\in X~|~~dist(x^{\prime},x)<\eta)\} \subset  X(=Sp(A))$$ and any point $y\in Sp(B)$, 
$$\sharp(Sp\phi_{y}\cap B_{\eta}(x))\geq\delta\cdot\sharp Sp\phi_{y},$$ counting multiplicity. If $\phi$ is not unital, we say that $\phi$ has $sdp(\eta,\delta)$ if the corresponding
unital homomorphism $\phi: A\longrightarrow \phi(\textbf{1}_{A})B\phi(\textbf{1}_{A})$ has property $sdp(\eta,\delta)$.
\end{numb}

\begin{numb}\label{1.28}
Set $P^{n}X=\underbrace{X\times X\times\cdots\times X}\limits_{n}/\thicksim$, where the equivalence relation $\thicksim$ is defined by $$(x_{1},x_{2},\cdots,x_{n})\thicksim(x^{\prime}_{1},x^{\prime}_{2},\cdots,x^{\prime}_{n})$$ if there is a permutation $\sigma$ of $\{1,2,\cdots,n\}$ such that
$x_{i}=x^{\prime}_{\sigma(i)}$ for each $1\leq i\leq n$. A metric $d$ on $X$ can be extended to a metric on $P^{n}X$ by $$d([x_{1},x_{2},\cdots,x_{n}],[x^{\prime}_{1},x^{\prime}_{2},\cdots,x^{\prime}_{n}])=\min\limits_{\sigma}\max\limits_{1\leq i\leq n}d(x_{i},x^{\prime}_{\sigma(i)}),$$
where $\sigma$ is taken from the set of all permutations, and $[x_{1},x_{2},\cdots,x_{n}]$ denote the equivalence class of $(x_{1},x_{2},\cdots,x_{k})$ in $P^{k}X$.
\end{numb}

\begin{numb}\label{1.29} Let $X$ be a metric space with metric $d$. Two $k$-tuple of (possible repeating) points $\{x_{1},x_{2},\cdots,x_{n}\}\subset X$ and
$\{x'_{1},x'_{2},\cdots,x'_{n}\}\subset X$ are said to {\bf be paired within} $\eta$ if there is a permutation $\sigma$ such that
$$d(x_{i},x^{\prime}_{\sigma(i)})<\eta,~~~ i=1,2,\cdots,k. $$ This is equivalent to the following statement. If one regards $[x_{1},x_{2},\cdots,x_{n}]$ and
$[x^{\prime}_{1},x^{\prime}_{2},\cdots,x^{\prime}_{n}]$ as points in $P^{n}X$, then $$d([x_{1},x_{2},\cdots,x_{n}],[x^{\prime}_{1},x^{\prime}_{2},\cdots,x^{\prime}_{n}])<\eta.$$
\end{numb}

\begin{numb}\label{1.30}
For $X=[0,1]$, let $P^{(n,k)}X$, where $n,k\in\mathbb{Z_{+}}\backslash \{0\}$, denote the set of $\frac{n}{k}$ elements from $X$, in which only 0 or 1 may appear
fractional times. That is,  each element in $X$ is of the form 
\begin{align}\label{rep1}
\quad\quad\quad\quad\quad\quad\quad\quad\quad
\{0^{\thicksim\frac{n_{0}}{k}},t_{1},t_{2},\cdots,t_{m},1^{\thicksim\frac{n_{1}}{k}}\}
\end{align}
with $0<t_{1}\leq t_{2}\leq\cdots\leq t_{m}<1$ and $\frac{n_{0}}{k}+m+\frac{n_{1}}{k}=\frac{n}{k}.$\\
An element in $P^{(n,k)}X$ can always be written as 
\begin{align}\label{rep2}
\quad\quad\quad\quad\quad\quad\quad\quad\quad
\{0^{\thicksim\frac{k_{0}}{k}},t_{1},t_{2},\cdots,t_{i},1^{\thicksim\frac{k_{1}}{k}}\}, 
\end{align}
where $0\leq k_{0}<k$, $0\leq k_{1}<k$, $0\leq t_{1}\leq t_{2}\leq\cdots\leq t_{i}\leq1$ and $\frac{k_{0}}{k}+i+\frac{k_{1}}{k}=\frac{n}{k}.$ (Here $t_{i}$
could be 0 or 1.) In the above representations \ref{rep1} and \ref{rep2}, we know that 
$$k_{0}\equiv n_{0}\mbox{ (mod }\;k) \mbox{ and }k_{1}\equiv n_{1}\mbox{ (mod }\;k).$$
 Let $$y=[0^{\thicksim\frac{k_{0}}{k}},t_{1},t_{2},\cdots,t_{i},1^{\thicksim\frac{k_{1}}{k}}]\in P^{(n,k)}X$$ and $$y^{\prime}=[0^{\thicksim\frac{k^{\prime}_{0}}{k}},t^{\prime}_{1},t^{\prime}_{2},\cdots,t^{\prime}_{i},1^{\thicksim\frac{k^{\prime}_{1}}{k}}]\in P^{(n,k)}X,$$ with $k_{0},k_{1},k^{\prime}_{0},k^{\prime}_{1}\in\{0,1,\cdots,k-1\}.$

We define $dist(y,y')$ as the following: if $k_{0}\neq k^{\prime}_{0}$ or $k_{1}\neq k^{\prime}_{1}$, then $dist(y,y^{\prime})=1$; if $k_{0}=k^{\prime}_{0}$ and $k_{1}= k^{\prime}_{1}$ (consequently $i=i^{\prime}$), then $$dist(y,y^{\prime})=\max\limits_{1\leq j\leq i}|t_{j}-t^{\prime}_{j}|,$$ as we order the $\{t_{j}\}$ and $\{t^{\prime}_{j}\}$ as
$t_{1}\leq t_{2}\leq\cdots\leq t_{i}$ and $t^{\prime}_{1}\leq t^{\prime}_{2}\leq\cdots\leq t^{\prime}_{i}$, respectively. 

Note that $P^{(n,1)}X=P^{n}X$ with the same metric. Let $\phi,\varphi: I_{k}\longrightarrow M_{n}(\mathbb{C})$ be two unital homomorphisms. Then $Sp\phi$ and
$Sp\psi$ define two elements in $P^{(n,k)}[0,1]$. We say that $Sp\phi$ and $Sp\psi$ can be paired within $\eta$, if $dist(Sp\phi,Sp\psi)<\eta$. 

Note that
if $dist(Sp\phi,Sp\psi)<1$, then $KK(\phi)=KK(\psi)$.
\end{numb}

\begin{numb}\label{specrestriction} Let $A=PM_{k}(C(X))P$, or $M_l(I_k)$ and $X_{1}\subset Sp(A)$ be a closed subset---that is,  $X_{1}$ is a closed subset of $X$ or of $[0,1]$. We define $A|_{X_{1}}$ to be the quotient algebra $A/I$, where $I=\{f\in A, f|_{X_{1}}=0\}.$
Evidently $Sp(A|_{X_{1}})=X_{1}$. 

If $B=QM_{k}(C(Y))Q$, $\phi: A\longrightarrow B$ is a homomorphism, and $Y_{1}\subset Sp(B)(=Y\;or\;[0,1])$ is a closed subset, then we use $\phi|_{Y_{1}}$ to denote the composition. $$\phi|_{Y_{1}}: A\xrightarrow{\phi} B\rightarrow B|_{Y_{1}}.$$
If $Sp(\phi|_{Y_{1}})\subset X_{1}\cup X_{2}\cup\cdots\cup X_{k}$, where $ X_{1},X_{2},\cdots,X_{k}$ are mutually disjoint closed subsets of $X$, then the
homomorphism $\phi|_{Y_{1}}$ factors as
$$A\longrightarrow A|_{X_{1}\cup X_{2}\cup\cdots\cup X_{n}}=\bigoplus\limits^{n}\limits_{i=1}A|_{X_{i}}\longrightarrow B|_{Y_{1}}.$$
We will use $\phi|^{X_{i}}_{Y_{1}}$ to denote the part of $\phi|_{Y_{1}}$ corresponding to the map $A|_{X_{i}}\longrightarrow B|_{Y_{1}}.$ Hence
$\phi|_{Y_{1}}=\bigoplus\limits_{i}\phi|^{X_{i}}_{Y_{1}}.$
\end{numb}

\section{Decomposition Theorem I}

In this section, we will prove the following theorem.

\begin{theorem}\label{THM} Let $F\subset I_{k}$ be a finite set, $\varepsilon > 0$.
There is an $\eta>0$, satisfying that if $$\phi: I_{k}\rightarrow PM_{\bullet}(C(X))P~~(dim(X)\leq 2)$$ is
a unital homomorphism such that for any $x\in X$,
\begin{center}
$\sharp(Sp\phi_{x}'\cap[0,\frac{\eta}4])\geq k$ and $\sharp(Sp\phi_{x}'\cap[1-\frac{\eta}4],1])\geq k,$
\end{center}
where
$$\phi':C[0,1]\xrightarrow{\imath} I_{k}\xrightarrow{\phi}PM(C(X))P,$$
then there are three mutually orthogonal projections ~$$Q_{0},~Q_{1},~P_{1}~\in PM_{\bullet}(C(X))P$$with$$Q_{0}+Q_{1}+P_{1}=P$$
and a unital homomorphism $$\psi_{1}:~M_{k}(C[0,1])\rightarrow P_{1}M_{\bullet}(C(X))P_{1}$$ such that\\
(1) write $\psi(f)={f}(\underline{0})Q_{0}+{f}(\underline{1})Q_{1}+(\psi_{1}\circ\imath)(f),$
then
$$\parallel\phi(f)-\psi(f)\parallel<\varepsilon$$
for all $f\in F\subset I_{k}\subset M_{k}(C[0,1]),$ and
 \\
(2) $rank(Q_{0})\leq k$ and $rank(Q_{1})\leq k$.
\end{theorem}

We will divide the proof into several steps.

\begin{numb}\label{5.2} Let $\eta>0$ (and $\eta<1$) be such that if $\mid t-t'\mid <\eta$, then $\parallel f(t)-f(t')\parallel< \frac{\varepsilon}{6}$ for all $f\in F$.
We will prove that this $\eta$ is as desired. Let a unital homomorphism $\phi:I_{k}\rightarrow PM_{\bullet}C(X)P$ satisfy that $\sharp(Sp\phi_{x}\cap[0,\frac{\eta}{4}])\geq k$ and $\sharp(Sp\phi_{x}\cap[1-\frac{\eta}{4},1])\geq k$ for each $x\in X$, we will prove such $\phi$
has the decomposition as desired.
\end{numb}

\begin{numb}\label{5.3} Let $rank(P)=n$. And let $e_{i,j}\in M_{n}(\mathbb{C})$ be the matrix units. For any closed set $Y\subset [0,1]$, define
$h_{Y}\in C[0,1]\subset I_{k}$ (considering C[0,1] as in the center of $I_{k})$ as

$$ h_{Y}(t)= \begin{dcases}
  1, & if\; t\in Y\\
   1-\frac{12n}{\eta}dist(t,x), & if \; dist(t,x)\leq\frac{\eta}{12n}\\
  0, & if \; dist(t,x)\geq\frac{\eta}{12n}~~.
   \end{dcases}$$

Define $H'=\{h_{Y}\mid Y$ is closed$\}\cup\{h_{Y}e_{ij}~|~~Y\subset[\frac{\eta}{12n},1-\frac{\eta}{12n}]$ is closed$\}$. Note that for a closed set $Y\subset[\frac{\eta}{12n},1-\frac{\eta}{12n}]$, ~~$h_{Y}(0)=h_{Y}(1)=0$, and therefore $h_{Y}e_{ij}\in I_{k}$. Note also that the family $H'$ is equally continuous. There is a finite set $H\subset H'$ satisfying that for any $h'\in H'$, $\exists h\in H$ such that $$\parallel h-h'\parallel\leq\frac{\varepsilon}{12(n+1)^{2}}~~.$$

For finite set $H\cup F$, $\varepsilon>0$, and $\phi:I_{k}\rightarrow PM_{\bullet}(C(X))P$, there is a $\tau>0$ such that the following are true:\\
(a) For $x,x'\in X$ with $dist(x',x)<\tau$, $Sp\phi|_{x}$ and $Sp\phi|_{x'}$ can be paired within
$\frac{\eta}{24n^{2}}$. This is equivalent to the  condition that $Sp\phi'|_{x}$ can be paired with $Sp\phi'|_{x'}$ to within $\frac{\eta}{24n^{2}}$ (since $KK(\phi|_{x})=KK(\phi|_{x'})$),  where $\phi'=\phi\circ \imath$ is as the above.\\
(b) For $x,x'\in X$ with $dist(x',x)<\tau$,$$\|\phi(h)(x)-\phi(h)(x')\|\leq\frac{\varepsilon}{12(n+1)^{2}}~,$$
regarding $\phi(h)(x)\in P(x)M_{\bullet}(\mathbb{C})P(x)\subset M_{\bullet}(\mathbb{C})$ and $\phi(h)(x')\in P(x')M_{\bullet}(\mathbb{C})P(x')\subset M_{\bullet}(\mathbb{C})$. In particular, $\|P(x)-P(x')\|<\frac{\varepsilon}{12(n+1)^{2}}$ since $1\in H$.
\end{numb}

\begin{numb}\label{5.4} Choose any simplicial decomposition on $X$ such that for any simplex $\Delta\subset X$, the set
\begin{center}
$Star(\triangle)=\cup\{\stackrel{\circ}{\Delta'}|\Delta'$ is a simplex of $X$ with $\Delta'\cap\Delta\neq\emptyset\}$
\end{center}
has diameter at most $\frac{\tau}{2}$, where $\mathop{\Delta^{\prime}}\limits^{\circ}$ is the interior of the simplex $\Delta^{\prime}.$
\end{numb}

\begin{numb}\label{5.5} 
We will construct the homomorphism $\psi:I_{k}\rightarrow PM_{\bullet}(C(X))P$ which is of the form $$\psi(f)=\underline{f}(0)Q_{0}+\underline{f}(1)Q_{1}+\psi_{1}(f)$$ as described in the theorem. Our construction will be carried out simplex by simplex. 

First, define the restriction of map $\psi$ to $PM_{\bullet}(C(X))P|_{v}=P(v)M_{\bullet}(\mathbb{C})P(v)$ for each vertex $v\in X$. The homomorphism is denoted by $$\psi|_{\{v\}}:I_{k}\rightarrow P(v)M_{\bullet}(\mathbb{C})P(v).$$ 
(Here and below, we refer the reader to \ref{specrestriction} for the notation $\psi|_{X_1}$ for a subset $X_1\subset X$.)

Next, we will define, for each 1-simplex $[a,b]\subset X$, the homomorphisms $$\psi|_{[a,b]}: I_{k}\rightarrow P|_{[a,b]}M_{\bullet}(C([a,b]))P|_{[a,b]}$$ which will give the same maps as the previously defined maps $\psi|_{\{a\}}$ and $\psi|_{\{b\}}$ on the boundary $\{a,b\}$. Finally, we will define, for each 2-simplex $\Delta\subset X$, the homomorphism $$\psi|_{\Delta}: I_{k}\rightarrow P|_{\Delta}M_{\bullet}(C(\Delta))P|_{\Delta}$$
such that $\psi|_{\partial\triangle}$ should be the  same as what previously defined.
\end{numb}

\begin{numb}\label{5.6}  For each simplex $\Delta$ of any dimension, let $C_{\Delta}$ denote the center of the simplex. That is, if $\Delta$ is a vertex $v$, then $C_{\Delta}=v$; if $\Delta$ is a 1-simplex identified with [a,b], then $C_{\Delta}=\frac{a+b}{2}$; and if $\Delta$ is a 2-simplex identified with a triangle in $\mathbb{R}^{2}$ with vertices $\{a,b,c\}\subseteq \mathbb{R}^{2}$, then $C_{\Delta}=\frac{a+b+c}{3}\in\mathbb{R}^{2}$ which is barycenter
of $\Delta$.
\end{numb}

\begin{numb} \label{partition}
According to each simplex $\Delta$ (of possible dimensions 0, 1, or 2), we will divide  the set $Sp\phi'|_{\Delta}\subset[0,1]$ into pieces, where  $\phi':C[0,1]\hookrightarrow I_{k}\xrightarrow{\phi}PM_{\bullet}C(X)P$.
(Recall $Sp\phi'|_{x}=\{Sp\phi|_{x}\}^{\sim k}$, and $Sp\phi'|_{x}$ has no fractional multiplicity). So for each $x\in X$,
\begin{center}
$Sp\phi'|_{x}=n=rank(P)$ (counting multiplicity).
\end{center}
If we order $Sp\phi'|_{x}$ as $$0\leq\lambda_{1}(x)\leq\lambda_{2}(x)\leq\cdots\leq\lambda_{n}(x)\leq1,$$
then all functions $\lambda_{i}$ are continuous functions. By path connectedness of simplex $\Delta$,  the set $Sp\phi|_{\Delta}$ can be written as $$Sp\phi|_{\Delta}=[a_{0},b_{0}]\cup[a_{1},b_{1}]\cup\cdots\cup[a_{k'-1},b_{k'-1}]\cup[a_{k'},b_{k'}]$$
with $$0\leq a_{0}\leq b_{0}<a_{1}\leq b_{1}<a_{2}\leq b_{2}<\cdots<a_{k'-1}\leq b_{k'-1}<a_{k'}\leq b_{k'}\leq1.$$
(Note that, if $a_i=b_i$, then $[a_i,b_i]=\{a_i\}$ is a degenerated interval.)  

We will group the above intervals into groups $T_{0}\cup T_{1}\cup\cdots \cup T_{last}$ such that $Sp\phi|_{\Delta}=\cup T_{j}$, with the condition that for any $\lambda\in T_{j}, \mu\in T_{j+1}$, we have $\lambda<\mu$, according to the following procedure:\\
(i) $Sp\phi|_{\Delta}\cap[0,\frac{\eta}{4}+\frac{\eta}{12n}]\subset T_{0}$, that is,  all the above intervals $[a_{i},b_{i}]$ with $a_{i}\leq\frac{\eta}{4}+\frac{\eta}{12n}$ should be in the group $T_{0}$; and $Sp\phi|_{\Delta}\cap[1-(\frac{\eta}{4}+\frac{\eta}{12n}),1]\subset T_{last}$, that is all $[a_{i},b_{i}]$ with $b_{i}\geq1-(\frac{\eta}{4}+\frac{\eta}{12n})$ will be grouped into the last group $T_{last}$;\\
(ii) If $a_{i}-b_{i-1}\leq \frac{\eta}{12n}$, then $[a_{i-1},b_{i-1}]$ and $[a_{i},b_{i}]$ are in the same group, say $T_{j}$;\\
(iii) If $a_{i}-b_{i-1}>\frac{\eta}{12n}$,  $a_{i}>\frac{\eta}{4}+\frac{\eta}{12n}$ and $ b_{i-1}<1-(\frac{\eta}{4}+\frac{\eta}{12n})$, then $[a_{i-1},b_{i-1}]$ and $[a_{i},b_{i}]$ are in different groups, say,  $T_{j}$ and $T_{j+1}$.

Denote $T_{last}$ by $T_{l_{\Delta}}$ (i.e., $l_{\Delta}=last$) --- if there is no confusion, we will call $T_{l_{\Delta}}$ by $T_{l}$. Let $t_{0}=0$, $s_0=\max\{\frac{\eta}{4}, \max T_0\}$, $t_l=\min\{1-\frac{\eta}{4}, \min T_l\}$, $s_{l}=1$; and for $1<i<l$, let $t_{i}=\min T_{i}$, and  $s_{i}=\max T_{i}$. Then $T_{i}\subset[t_{i},s_{i}]$. With the above notation, we have the following lemma.
\end{numb}

\begin{lemma} \label{5.8}
 With the above notation, we have the following\\
(a) length $[t_{0},s_{0}]\leq\frac{\eta}{4}+\frac{\eta}{6}$;\\
(b)  length $[t_{l},s_{l}]\leq\eta/4+\eta/6$;\\ 
(c) length $[t_{i},s_{i}]\leq\eta/6$ for $i\in\{1,2,\cdots,l-1\}$;\\
(d) $t_{i+1}-s_{i}>\frac{\eta}{12n}$ for $i\in\{0, 1,2,\cdots,l-1\}$.
\end{lemma}

\begin{proof} From (ii) of \ref{partition}, we know that $\min T_{i+1}-\max T_i> \frac{\eta}{12n}$; and from (i), we know that $\min T_1> \frac{\eta}{4}+\frac{\eta}{12n}$ and $\max T_{l-1}< 1-(\frac{\eta}{4}+\frac{\eta}{12n})$. Hence (d) holds.

The following fact is well known.

Fact: For any two sequences $0\leq\lambda_{1}\leq\lambda_{2}\leq\cdots\leq\lambda_{n}\leq1$ and $0\leq\mu_{1}\leq\mu_{2}\leq\cdots\leq\mu_{n}\leq1$, $\{\lambda_{i}\}_{i=1}^{n}$ and $\{\mu_{i}\}_{i=1}^{n}$ can be paired within $\sigma$ if and only if $|\lambda_{i}-\mu_{i}|<\sigma$ for all $i\in\{1,2,\cdots,n\}$.

Note that $\Delta$ is path connected and $Sp\phi|_{\Delta}=\bigcup\limits_{i=1}\limits^{k'}[a_{i},b_{i}]$ with $[a_{i},b_{i}]\cap[a_{j},b_{j}]=\emptyset$ if $i\neq j$. We conclude that for any $z,z'\in\Delta$ and $i$, $$\sharp(Sp\phi|_{z}\cap[a_{i},b_{i}])=\sharp(Sp\phi|_{z'}\cap[a_{i},b_{i}])$$
counting multiplicity. In our construction, we know that $Sp\phi|_{z}$ and $Sp\phi|_{z'}$ can be paired within $\eta/24n^{2}$, using the above mentioned fact. We know also that
\begin{center}
$Sp\phi|_{z}\cap[a_{i},b_{i}]$ ~~~~~and~~~~ $Sp\phi|_{z'}\cap[a_{i},b_{i}]$
\end{center}
can be paired within $\eta/24n^{2}$. Consequently $$[a_{i},b_{i}]\subset_{\eta/24n^{2}}[a_{i},b_{i}]\cap Sp\phi|_{C_{\Delta}},$$
where $C_{\Delta}$ is the center of simplex $\Delta$. Note that $Sp\phi|_{C_{\Delta}}\cap[a_{i},b_{i}]$ is a finite set with at most $n$ points in $[0,1]$ and $\eta/24n^{2}$-neighborhood of each point is a closed interval of length at most $(\eta/24n^{2})\cdot2=\eta/12n^{2}$. Hence we have
$$\mbox{length}[a_{i},b_{i}]\leq(\eta/12n^{2})\cdot n=\eta/12n.$$

Furthermore, each $T_{j}$ contains at most $n$ intervals $[a_{i},b_{i}]$.  And for each consecutive pair of intervals in $T_{j}$ ($0<j<l$),  we have  
$$[a_{i},b_{i}]\cup[a_{i+1},b_{i+1}]\subset\big(\frac{\eta}{4}+\frac{\eta}{12n},1-(\frac{\eta}{4}+\frac{\eta}{12n})\big)$$ 
and the distance between them $a_{i+1}-b_{i}\leq\eta/12n$. That is,  the gap between them is at most $\eta/12n$. Hence for each $i\in\{1,2,\cdots,l-1\}$, the length of $[t_{i},s_{i}]$ is at most $$n\cdot\frac{\eta}{12n}+(n-1)\cdot\frac{\eta}{12n}<\eta/6$$
(at most $n$ possible intervals and $n-1$ gaps). 

Also,
$$\mbox{length} [t_{0},s_{0}]<\frac{\eta}{4}+\frac{\eta}{6}$$
and$$\mbox{length} [t_{l},s_{l}]<\frac{\eta}{4}+\frac{\eta}{6}~~.$$\\
\end{proof}

\begin{numb} \label{5.9}
For each simplex $\Delta$ with face $\Delta'\subset\Delta$, we use $T_{i}(\Delta)$ and $T_{j}(\Delta')$ to denote the sets $[t_{i}(\Delta),s_{i}(\Delta)]$ or $[t_{j}(\Delta'),s_{j}(\Delta')]$ as in \ref{partition}, corresponding to $\Delta$ and $\Delta'$. 
Then evidently, the decomposition $$Sp\phi|_{\Delta'}=\bigcup\limits_{j}(T_{j}(\Delta')\cap Sp\phi|_{\Delta'}),$$ is a refinement of the decomposition $Sp\phi|_{\Delta}=\cup(T_{i}(\Delta)\cap Sp\phi|_{\Delta})$--- that is, if two elements $\lambda,~\mu\in Sp\phi|_{\Delta'}$ are in the  set $T_{j}(\Delta')$  for  a same index $j$, then they are in the  set $T_{i}(\Delta)$ for a same index $i$.
\end{numb}

\begin{numb} \label{MAP DECOM} For each simplex $\Delta$, consider the homomorphism 
$$\phi:I_{k}\rightarrow PM_{\bullet}(C(\Delta))P=A|_{\Delta}~.$$
Since $Sp\phi|_{\Delta}\subset\bigcup_{j=0}^{l}T_{j}(\Delta)=\bigcup\limits_{j=1}\limits^{l}[t_{j},s_{j}]$,
$\phi$ factors through as $$I_{k}\rightarrow\oplus_{j=0}^{l}I_{k}|_{[t_{j},s_{j}]}\xrightarrow{\oplus\phi_{j}}PM_{\bullet}(C(\Delta))P.$$
Let $P_{j}(x)=\phi_{j}(\textbf{1}_{k}|_{[t_{j},s_{j}]})(x)$ for each $x\in\Delta$. Then $P_{j}(x)$ are mutually orthogonal projections  satisfying $$\sum\limits_{j=0}\limits^{l}P_{j}(x)=P(x).$$ By the assumption  of Theorem \ref{THM}, we have $rank(P_{0})\geq k$ and $rank(P_{l})\geq k$.
\end{numb}

\begin{numb}\label{0-dim} Now we define $\psi:I_{k}\rightarrow A|_{\Delta}$ simplex by simplex,  starting with vertices --- the zero dimensional simplices.

Let $v\in X$ be a vertex. As in \ref{partition}, we write
$$Sp\phi|_{\{v\}}=\bigcup\limits_{i=0}\limits^{l}[t_{i},s_{i}]\bigcap Sp\phi|_{\{v\}},$$
where $0=t_{0}<s_{0}<t_{1}\leq s_{1}<\cdots<t_{l-1}\leq s_{l-1}<t_{l}<s_{l}=1$, with $$[0,\eta/4]\subset[t_{0},s_{0}]\subset[0,\eta/2],$$ $$[1-\eta/4,1]\subset[t_{l},s_{l}]\subset[1-\eta/2,1],$$ $$0\leq s_{i}-t_{i}<\eta/6~~~~\mbox{ for each}~~~i\in \{1,2,\cdots, l-1\}, ~~and$$
$$t_{i+1}-s_{i}>\eta/12n ~~~\mbox { for each } i\in \{0,1,2,\cdots, l-1\}.$$

Recall that $\phi|_{\{v\}}: I_{k}\rightarrow P(v)M_{\bullet}(\mathbb{C})P(v)$ (as in \ref{MAP DECOM}) can be written as $$\phi|_{\{v\}}=diag(\phi_{0},\phi_{1},\cdots,\phi_{l}): I_{k}\rightarrow\mathop{\bigoplus}\limits_{i=0}^{l}P_{i}M_{\bullet}(\mathbb{C})P_{i}\subset P(v)M_{\bullet}(\mathbb{C})P(v),$$
where $\phi_{i}\!=\!\phi|_{\{v\}}^{[t_i,s_i]}:I_{k}\rightarrow I_{k}|_{[t_{i},s_{i}]}\rightarrow P_{i}M_{\bullet}(\mathbb{C})P_{i}$ and $P(v)=\sum_{i=0}^{l}P_{i}$.  (Here and below, we refer the reader to \ref{specrestriction} for the notation $\phi|_{X_1}^{Z_j}$ ($X_1\subset X$),  which makes sense, provided that $Sp(\phi|_{X_1}) \subset \bigcup_j Z_j$, where $\{Z_j\}$ are mutually disjoint closed subsets of the spectrum of the domain algebra of $\phi$.)

From now on, we will use  $diag_{0\leq i \leq l}~( \phi_i)$ to denote $diag(\phi_{0},\phi_{1},\cdots,\phi_{l})$.

Define $\psi_{i}:I_{k}|_{[t_{i},s_{i}]}\rightarrow P_{i}M_{\bullet}(\mathbb{C})P_{i}$ by
\begin{center}
$\psi_{i}=\phi_{i}$~~~~~ if~~~~~ $1\leq i\leq l-1$,
\end{center}
(That is, we do not modify $\phi_{i}$ for $1\leq i\leq l-1.)$
For $i=0$ (the case $i=l$ is similar) we do the following modification. There is a unitary $u\in M_{\bullet}(\mathbb{C})$ such that

$$\phi_{0}(f)(v)=u\left(
   \begin{array}{ccc}
     \left(
        \begin{array}{cccc}
         {f}(\underline0) & \ & \ & \ \\
          \ &\hspace{-.1in} {f}(\underline0) & \ & \ \\
          \ & \ & \hspace{-.2in} \ddots & \ \\
          \ & \ & \ & \hspace{-.1in} {f}(\underline0) \\
        \end{array}
      \right)_{\!\!\!j\times j}
      & \ & \ \\
     \ &\hspace{-.5in} \left.
           \begin{array}{cccc}
             f(\xi_{1}) & \ & \ & \ \\
             \ & f(\xi_{2}) & \ & \ \\
             \ & \ & \ddots & \ \\
             \ & \ & \ & f(\xi_{\bullet\bullet}) \\
           \end{array}
         \right.
      & \ \\
     \ & \ & \hspace{-.3in}\left.
               \begin{array}{ccc}
                 0 & \ & \ \\
                 \ & \ddots & \ \\
                 \ & \ & 0 \\
               \end{array}
             \right.
      \\
   \end{array}
 \right)
 u^{\ast},$$
where $\xi_{i}\in(0,s_{1}]$, $0<\xi_{1}\leq\xi_{2}\leq\cdots\leq\xi_{\bullet\bullet}\leq s_{1}.$ Or write it as 
$$\phi_0(f)(v)=u\mbox{diag}(f(\underline{0})^{\sim j}, f(\xi_1), f(\xi_2),\cdots,f(\xi_{\bullet\bullet}),0,\cdots,0) u^*.$$
If $0<j\leq k$ then we do not do any modification and just let $\psi_{0}=\phi_{0}$. If $j>k$, then write $j=kk'+j'$ with $0<j'\leq k$,  choose $\xi'\in(0,\xi_{1})$,  and define
$$\psi_0(f)(v)=u\mbox{diag}(f(\underline{0})^{\sim j'}, f(\xi')^{\sim k'}, f(\xi_1), f(\xi_2),\cdots,f(\xi_{\bullet\bullet}),0,\cdots,0) u^*.$$
That is, change $kk'$ terms of ${f}(\underline0)$ in the diagonal of the definition of $\phi_{0}$ to $k'$ terms of the form $f(\xi')$.
If $j=0$, then we change $\xi_{1}$ to 0, that is,
$$\psi_0(f)(v)=u\mbox{diag}(f(\underline{0})^{\sim k}, f(\xi_2),\cdots,f(\xi_{\bullet\bullet}),0,\cdots,0) u^*.$$

Since $|\xi'-0|<\frac{\eta}{2}$ and $|\xi_{1}-0|<\frac{\eta}{2}$, we have $\|\phi_{0}(f)-\psi_{0}(f)\|<\frac{\varepsilon}{6}$, for all $f\in F$ (see \ref{5.2}).
We modify  $\phi_{l}$ in a similar way  to define $\psi_{l}$. Let $$\psi|_{\{v\}}=diag(\psi_{0},\psi_{1},\cdots,\psi_{l}):I_{k}\rightarrow P(v)M_{\bullet}(\mathbb{C})P(v),$$
where $\psi_{i}\!=\!\psi|_{\{v\}}^{[t_i,s_i]}$.  Then $\|\phi(f)-\psi(f)\|<\frac{\varepsilon}{6}$ for all $f\in F$.
\end{numb}

\begin{remark}\label{5.12} Let us emphasize that the homomorphisms $\psi_i$ are the same as $\phi_i$ for \\ $i\in\{1,2,\cdots, l_{\{v\}}-1\}$. But we do  modify $\phi_0$ and $\phi_l$ ($l= l_{\{v\}}$) to get $\psi_0$ and $\psi_l$. 

Also, we have  $$Sp(\psi_{0})\subset[0,s_{0}]~~~~\mbox{and}~~~~Sp(\psi_{l_{\{v\}}})\subset[t_{l_{\{v\}}},1].$$ Furthermore, $\psi_{i}(1)=\phi_{i}(1)$ for any $i$, and consequently $\psi(1)=\phi(1).$
\end{remark}

\begin{numb}\label{5.13} Now consider 1-simplex $\Delta\!=\![a,b]\subset X$. We  need to define $\psi|_{\Delta}=\psi|_{[a,b]}$ from previously defined  $\psi|_{\{a\}}$ and $\psi|_{\{b\}}$. According to \ref{partition},  write $Sp\phi|_{\Delta}=\bigcup\limits_{j=1}\limits^{l_{\Delta}}Sp\phi|_{\Delta}\cap T_{j}({\Delta})$ with $T_{0}({\Delta})=[0, s_{0}(\Delta)]$ and $T_{l_{\Delta}}({\Delta})=[t_{l_{\Delta}}(\Delta),1]$. Recall that in the definition of $\psi|_{\{a\}}$, $\psi|_{\{b\}}$, we use the decomposition $$\phi|_{\{a\}}=diag_{1\leq j\leq l_{\{a\}}}(\phi|_{\{a\}}^{T_{j}{(\{a\})}})$$ and $$\phi|_{\{b\}}=diag_{1\leq j\leq l_{\{b\}}}(\phi|_{\{b\}}^{T_{j}{(\{b\})}})$$ and only modified $\phi_{0}=\phi|_{\{a\}}^{[0,s_{0}\{a\}]}$ (or  $\phi|_{\{b\}}^{[0,s_{0}\{b\}]}$) and $\phi_{l\{a\}}=\phi|_{\{a\}}^{[t_{l_{\{a\}}}(\{a\}), 1]}$ (or  $\phi|_{\{b\}}^{[t_{l_{\{b\}}}(\{b\}), 1]}$).

For $\Delta=[a,b]$, let us consider the decomposition
$$\phi|_{\Delta}=\bigoplus_{j=1}^{l_{\Delta}}\phi|_{\Delta}^{[t_{j}(\Delta),s_{j}(\Delta)]}.$$
From the above, we know that for any $0<j<l_{\Delta}$, the definition of $\psi|_{\{a\}}^{[t_{j}(\Delta),s_{j}(\Delta)]}$  is the same as $(\phi|_{\Delta}^{[t_{j}(\Delta),s_{j}(\Delta)]})|_{\{a\}}$, since the decomposition 
$$Sp\phi|_{\{a\}}=\bigcup\limits_{j=1}^{l_{\{a\}}}T_{j}(\{a\})\cap Sp\phi|_{\{a\}}$$ 
is finer than the decomposition 
$$Sp\phi|_{\{a\}}=\bigcup\limits_{j=1}^{l_{\Delta}}T_{j}(\Delta)\cap Sp\phi|_{\{a\}}$$ 
(see \ref{5.9}) and  only partial maps involving $[0,s_{1}\{a\}]$ ($\subset[0,s_{1}(\Delta)]$) and $[t_{l_{\{a\}}}(\{a\}),1]$ ($\subset[t_{l_{\Delta}}(\Delta),1]$) are modified. The same is true for $\phi|_{\{b\}}$ and $\psi|_{\{b\}}$. Therefore,  we can define the partial maps $$\psi|_{\Delta}^{[t_{j}(\Delta),s_{j}(\Delta)]}=\phi|_{\Delta}^{[t_{j}(\Delta),s_{j}(\Delta)]}$$
for $0<j<l_{\Delta}$. 
The only parts need to be modified are $\phi|_{\Delta}^{[0,s_{0}(\Delta)]}$ and $\phi|_{\Delta}^{[t_{l}(\Delta),1]}$.
\end{numb}

\begin{numb}\label{5.14} Now denote $\phi|_{\Delta}^{[0,s_{0}(\Delta)]}(\Delta=[a,b])$ by $\phi_{0}$ and $\phi|_{\Delta}^{[t_{l}(\Delta),1]}$ by $\phi_{l}$, and $s_{0}(\Delta)$ by $s_{0}$, $t_{l(\Delta)}(\Delta)$ by $t_{l}$. Now we have two unital homomorphisms
$$\phi_{0}:I_{k}|_{[0,s_{0}]}\rightarrow P_{0}M_{\bullet}C(\Delta)P_{0}$$
and
$$\phi_{l}:I_{k}|_{[t_{l},1]}\rightarrow P_{l}M_{\bullet}C(\Delta)P_{l},$$
where $P_{0}$, $P_{l}$ are defined as in \ref{MAP DECOM}. We will do the modification of $\phi_{0}$ to get $\psi_{0}$ (the one for $\phi_{l}$ is completely the same).

We already have the definitions of $\psi_{0}|_{\{a\}}$ and $\psi_{0}|_{\{b\}}$. Note that $P_{0}\in M_{\bullet}(C(\Delta))$ can be written as $\phi(h_{[0,s_{0}]})$, where $h_{[0,s_{0}]}$ is the test function appeared  in \ref{5.3}, which is equal to $1$ on $[0,s_{0}]$ and $0$ on $[s_{0}+\frac{\eta}{12n},1]$.
(Note that $\phi(h_{[0,s_{0}]})$ is a projection since $Sp\phi\subset[0,s_{0}]\cup[t_{1},1]$ and $t_{1}>s_{0}+\frac{\eta}{12n}$.) Consequently,
 $$\|P_{0}(x)-P_{0}(y)\|<\frac{\varepsilon}{12(n+1)^{2}}$$
 for all $x,y\in[a,b]=\Delta$ (see (b) of \ref{5.3}).

There exists a unitary $W\in M_{\bullet}(C(\Delta))$ such that
\begin{center}
$P_{0}(x)=W(x)$ $\left(
                  \begin{array}{cc}
                  \left(
        \begin{array}{ccc}
                    1 & \ & \ \\
                    \ & \ddots & \  \\
                    \ & \ & 1  \\
                    \end{array}
      \right )_{\!\!\!rank(P_0)\times rank(P_0)}
        & \ \\
     \ & \hspace{-.5in}\left.
           \begin{array}{ccc}
                     0 & \ & \ \\
                    \ & \ddots & \ \\
                     \ & \ & 0 \\
                     \end{array}
         \right.
      \\
       \end{array}
                \right)$ $W^{\ast}(x)$,
\end{center}
for all $ x\in \Delta$
and $\displaystyle{\|W(x)-W(y)\|<\frac{\varepsilon}{6(n+1)^{2}}}$.

To define $$\psi_{0}:I_{k}|_{[0,s_{0}]}\rightarrow P_{0}M_{\bullet}(C(\Delta))P_{0},$$ it suffices to define $$AdW\circ\psi_{0}:I_{k}|_{[0,s_{0}]}\rightarrow M_{rank(P_{0})}(C(\Delta)),$$ since $$W^{\ast}P_{0}W=\left(
                                                                                                                    \begin{array}{cc}
                                                                                                                      \textbf{1}_{rank(P_{0})} & 0 \\
                                                                                                                      0 & 0 \\
                                                                                                                    \end{array}
                                                                                                                  \right)
.$$

Note that $$\sharp(Sp\widetilde{\psi}_{0}|_{\{a\}}\cap\{0\})=rank(P_{0})~~~~(mod\ k),$$ where $$\widetilde{\psi}_{0}:C[0,s_{0}]\hookrightarrow I_{k}|_{[0,s_{0}]}\xrightarrow{\psi_{0}}P_{0}(\{a\})M_{\bullet}(\mathbb{C})P_{0}(\{a\}).$$ 
(This is true since the multiplicities of all the spectra other than $0$ are multiples  of $k$ ). Similarly, $$\sharp(Sp\widetilde{\psi}_{0}|_{\{b\}}\cap\{0\})=rank(P_{0})~~~~(mod\ k).$$ Also,  from the definition of $\psi$ on the vertices (namely on $\{a\}$ and $\{b\}$) from \ref{0-dim}, we know that $$\sharp(Sp\widetilde{\psi}_{0}|_{\{b\}}\cap\{0\})=\sharp(Sp\widetilde{\psi}_{0}|_{\{a\}}\cap\{0\})\triangleq k'\leq k.$$

\end{numb}

\begin{lemma}\label{5.15}
Suppose that two unital homomorphisms $$\alpha',\alpha'': I_{k}|_{[0,s_{0}]}\rightarrow M_{rank(P_{0})}(\mathbb{C})$$
satisfy that $$0<\sharp(Sp\widetilde{\alpha}'\cap\{0\})=\sharp(Sp\widetilde{\alpha}''\cap\{0\})\leq k$$ counting multiplicity, where $\widetilde{\alpha}'$ (or $\widetilde{\alpha}''$) is the composition $$C[0,s_{0}]\hookrightarrow I_{k}|_{[0,s_{0}]}\xrightarrow{\alpha'}M_{rank(P_{0})}(\mathbb{C}))~~(\mbox{or}~~C[0,s_{0}]\hookrightarrow I_{k}|_{[0,s_{0}]}\xrightarrow{\alpha''}M_{rank(P_{0})}(\mathbb{C})),$$
then there is a homomorphism $$\alpha:I_{k}|_{[0,s_{0}]}\rightarrow M_{rank(P_{0})}(C[a,b]),$$
such that $0<\sharp(Sp\widetilde{\alpha}|_{t}\cap\{0\})\leq k$, for all $t\in[a,b]$ and $\alpha|_{\{a\}}=\alpha'$, $\alpha|_{\{b\}}=\alpha''$, where again $\widetilde{\alpha}$ is the composition $$C[0,s_{0}]\hookrightarrow I_{k}|_{[0,s_{0}]}\xrightarrow{\alpha}M_{rank(P_{0})}(C[a,b]).$$
\begin{proof}
 We can regard $[a,b]=[0,1]$. There are two unitaries $u,v\in M_{rank(P_{0})}(\mathbb{C})$, a number $k'\in\{1,2,\ldots,k\}$,  and two finite sequences of  numbers: $$0<\xi_{1}\leq\xi_{2}\leq\cdots\leq\xi_{\bullet}\leq s_{0}$$ $$0<\xi'\leq\xi_{2}'\leq\cdots\leq\xi_{\bullet}'\leq s_{0}$$ such that\\
$$\alpha'(f)=u\left(
   \begin{array}{cc}
     \left(
        \begin{array}{ccc}
         {f}( \underline0) & \ & \ \\
          \ & \ddots & \  \\
          \ & \ & {f}( \underline0)  \\
        \end{array}
      \right )_{\!\!\!k'\times k'}
      & \ \\
     \ & \hspace{-.5in}\left.
            \begin{array}{ccc}
               f(\xi_{1}) & \ & \ \\
                 \ & \ddots & \ \\
                 \ & \ & f(\xi_{\bullet}) \\
            \end{array}
          \right.
      \\
   \end{array}
 \right)
u^{\ast}$$
and$$\alpha''(f)=v\left(
   \begin{array}{cc}
     \left(
        \begin{array}{ccc}
          {f}( \underline0) & \ & \ \\
          \ & \ddots & \  \\
          \ & \ & {f}( \underline0)  \\
        \end{array}
      \right )_{\!\!\!k' \times k'}
      & \ \\
     \ &\hspace{-.5in} \left.
            \begin{array}{ccc}
               f(\xi_{1}') & \ & \ \\
                 \ & \ddots & \ \\
                 \ & \ & f(\xi_{\bullet}') \\
            \end{array}
          \right.
      \\
   \end{array}
 \right)
v^{\ast}.$$

Let $u(t)$, $0\leq t\leq\frac{1}{2}$ be any unitary path with $u(0)=u$, $u(\frac{1}{2})=v$. Define $\alpha$ as follows.

For $0\leq t\leq\frac{1}{2},$
\begin{center}
$\alpha(f)(t)=u(t){\small \left(
   \begin{array}{cc}
     \left(
        \begin{array}{ccc}
          {f}( \underline0) & \ & \ \\
          \ & \ddots & \  \\
          \ & \ & {f}( \underline0)  \\
        \end{array}
      \right)_{\!\!\!k'\times k'}
      & \ \\
     \ & \hspace{-.5in}\left.
            \begin{array}{cccc}
               f(\xi_{1}) & \ & \ & \ \\
               \ &  f(\xi_{2}) & \ & \ \\
               \ &  \ & \ddots & \ \\
               \ &  \ & \ & f(\xi_{\bullet}) \\
            \end{array}
          \right.
      \\
   \end{array}
 \right)}
u^{\ast}(t);$
\end{center}
and for $\frac{1}{2}\leq t\leq1, $
 $$\alpha(f)(t)=v{\small\left(
   \begin{array}{cc}
     \left(
        \begin{array}{ccc}
          {f}( \underline0) & \ & \  \\
                \ & \hspace{-.1in}\ddots & \ \\
                \ & \ & \hspace{-.1in}{f}( \underline0)  \\
        \end{array}
      \right )_{\!\!\!k'\times k'}
      & \ \\
     \ & \hspace{-.5in}\left.
           \begin{array}{ccc}
              f((2-2t)\xi_{1}+(2t-1)\xi_{1}') & \ & \ \\
                 \ & \hspace{-.3in}\ddots & \ \\
                \ & \ & \hspace{-.1in}f((2-2t)\xi_{\bullet}+(2t-1)\xi_{\bullet}') \\
           \end{array}
         \right.
      \\
   \end{array}
 \right)}
v^{\ast}.$$
Then $\alpha$ is a desired homomorphism.
\end{proof} 
\end{lemma}

\begin{numb}
Applying the above lemma, we can define $$\alpha:I_{k}|_{[0,s_{0}]}\rightarrow M_{rank(P_{0})}(C[a,b])$$
such that $$\imath\circ\alpha|_{\{a\}}=AdW(a)\circ\psi_{0}|_{\{a\}}$$ and $$\imath\circ\alpha|_{\{b\}}=AdW(b)\circ\psi_{0}|_{\{b\}},$$ where $\imath:M_{rank(P_{0})}(\mathbb{C})\rightarrow M_{\bullet}(\mathbb{C})$ is defined by $$\imath(A)=\left(
                                                                                                \begin{array}{cc}
                                                                                                  A & 0 \\
                                                                                                  0 & 0 \\
                                                                                                \end{array}
                                                                                              \right)
.$$
Define $$\psi_{0}:I_{k}|_{[0,s_{0}]}\rightarrow P_{0}M_{\bullet}(C(\Delta))P_{0}$$ by $\psi_{0}=AdW^*\circ(\imath\circ\alpha)$ --- that is, for any $t\in[a,b]=\Delta$,
$$\psi_{0}(f)(t)=W(t)\left(
                       \begin{array}{cc}
                         \alpha(f)(t) & 0 \\
                         0 & 0 \\
                       \end{array}
                     \right)
W^{\ast}(t).$$

As mentioned in \ref{5.13}, when we modify $\phi|_{[a,b]}$ to obtain $\psi|_{[a,b]}$, we only need to modify $\phi_{0}=\phi|_{[a,b]}^{[0,s_{0}]}$ and $\phi_{l}=\phi|_{[a,b]}^{[t_{l},1]}$. The modifications of $\phi_{l}$ to $\psi_{l}$ are the same as the one  from $\phi_{0}$ to $\psi_{0}$. Thus we have the definition of $\psi|_{[a,b]}=diag_{0\leq i\leq l}(\psi_{i})$.
\end{numb}

\begin{numb}\label{5.17} Let us estimate the difference of $\phi|_{[a,b]}$ and $\psi|_{[a,b]}$ on the finite set $F\subset I_{k}$. Note that $$\phi|_{[a,b]}=diag_{0\leq i\leq l}(\phi_{i}),\ \ \psi|_{[a,b]}=diag_{0\leq i\leq l}(\psi_{i})$$
and $\phi_{i}=\psi_{i}$, for $0< i <l$. So we only need to estimate $\|\phi_{0}(f)-\psi_{0}(f)\|$ and $\|\phi_{l}(f)-\psi_{l}(f)\|$.

Note that $\phi_{0}$ and $\psi_{0}$ are from $I_{k}|_{[0,s_{0}]}$ to $P_{0}M_{\bullet}(C[a,b])P_{0}$, where $P_{0}$ is as in \ref{5.14}. And both $AdW\circ\phi_{0}$ and $AdW\circ\psi_{0}$ can be regarded as $\imath\circ\phi'$ and $\imath\circ\psi'$ for
$$\phi',\psi':I_{k}|_{[0,s_{0}]}\rightarrow M_{rank(P_{0})}(C[a,b]),$$
where $$\imath:M_{rank(P_{0})}(C[a,b])\rightarrow M_{\bullet}(C[a,b])$$ is given by $$\imath(A)=\left(
                                                                                        \begin{array}{cc}
                                                                                          A & 0 \\
                                                                                          0 & 0 \\
                                                                                        \end{array}
                                                                                      \right)
.$$

Claim:  Let $\alpha: I_{k}|_{[0,s_{0}]}\rightarrow M_{rank(P_0)}(C[a,b])$ be any unital homomorphism. Then we have
$$\Bigg\|\alpha(f)-\left(
                                                                                                                      \begin{array}{ccc}
                                                                                                                       {f}( \underline0) & \ & \ \\
                                                                                                                        \ & \ddots & \ \\
                                                                                                                        \ & \ & {f}( \underline0) \\
                                                                                                                      \end{array}
                                                                                                                    \right)
_{\!\!\!rank(P_{0})}\Bigg\|\leq\sup\limits_{0<\xi\leq s_{0}}\|f(\xi)-f(0)\|.$$
In fact, for each $x\in[a,b]$, there exist $u_{x}\in U(M_{rank(P_{0})}(\mathbb{C}))$ and $k'\in\{1,2,\cdots,k\}$ and $0\leq\xi_{1}\leq\xi_{2}\leq\cdots\leq\xi_{\bullet\bullet}\leq s_{0}$ such that
$$\alpha(f)(x)=u_{x}\left(
   \begin{array}{cc}
     \left(
        \begin{array}{ccc}
          {f}( \underline0) & \ & \ \\
          \ & \ddots & \  \\
          \ & \ & {f}( \underline0)  \\
        \end{array}
      \right )_{\!\!\!k'\times k'}
      & \ \\
     \ & \hspace{-.5in}\left.
            \begin{array}{ccc}
               f(\xi_{1}) & \ & \ \\
                 \ & \ddots & \ \\
                 \ & \ & f(\xi_{\bullet\bullet}) \\
            \end{array}
          \right.
      \\
   \end{array}
 \right)
u_{x}^{\ast}.$$
It follows that 

$$\|\alpha(f)(x)-{f}( \underline0)\cdot\textbf{1}_{rank(P_{0})}\| $$
\begin{align*}
 &= \left\| u_{x} {~\small \left[\left(
    \begin{array}{cc}
     \left(
        \begin{array}{ccc}
         \hspace{-.1in}{f}( \underline0) & \ & \ \\
          \ & \hspace{-.2in}\ddots & \  \\
          \ & \ & \hspace{-.1in}{f}( \underline0)  \\
        \end{array}
      \right )_{\!\!\!k' \times k'}
      & \ \\
     \ &\hspace{-.5in} \left.
            \begin{array}{ccc}
               f(\xi_{1}) & \ & \ \\
                 \ & \hspace{-.3in}\ddots & \ \\
                 \ & \ & \hspace{-.2in}f(\xi_{\bullet\bullet}) \\
            \end{array}
          \right.
      \\
   \end{array}
 \right)-\left(
   \begin{array}{cc}
     \left(
        \begin{array}{ccc}
          \hspace{-.1in}{f}( \underline0) & \ & \ \\
          \ & \hspace{-.1in}\ddots & \  \\
          \ & \ & \hspace{-.1in}{f}( \underline0)  \\
        \end{array}
      \right)_{\!\!\!k' \times k'}
      & \ \\
     \ & \hspace{-.5in}\left.
            \begin{array}{ccc}
               f(0) & \ & \ \\
                 \ & \hspace{-.3in}\ddots & \ \\
                 \ & \ & \hspace{-.2in}f(0) \\
            \end{array}
          \right.
      \\
   \end{array}
 \right)\right]}u^*_{x}\right\|\\
& { =\left\|\left(
     \begin{array}{cc}
       \left(
          \begin{array}{cccc}
            0 & \ & \ & \  \\
            \ & 0 & \ & \  \\
            \ & \ & \ddots & \  \\
            \ & \ & \ & 0  \\
          \end{array}
        \right)_{\!\!\!k'\times k'}
        & \ \\
       \ & \hspace{-.5in}\left.
              \begin{array}{ccc}
                f(\xi_{1})-f(0) & \ & \ \\
                                  \ & \ddots & \ \\
                                 \ & \ & f(\xi_{\bullet\bullet})-f(0) \\
              \end{array}
            \right.
        \\
     \end{array}
   \right)
           \right\|} \\
          & \leq\sup\limits_{0\leq\xi\leq s_{0}}\|f(\xi)-f(0)\parallel.
  \end{align*}
Thus, the claim is true.

  $
  $

It follows from the claim that
$$\|\phi(f)(t)-\psi(f)(t)\|\leq2\max\big(\sup\limits_{0\leq\xi\leq s_{0}}\|f(\xi)-f(0)\|,\sup\limits_{t_{l}\leq\xi\leq 1}\|f(\xi)-f(1)\|\big)\leq2\cdot\frac{\varepsilon}{6}$$
\\for all $t\in[a,b]$, and $f\in F$, as $|s_{0}-0|<\frac{\eta}{2}$ and $|t_{l}-1|<\frac{\eta}{2}$. 
Hence we have the definition of $\psi$ on the 1-skeleton $X^{(1)}\subset X$ satisfying
$$\|\phi(f)(t)-\psi(f)(t)\|<\frac{\varepsilon}{3}$$
for all $t\in X^{(1)}$ and $f\in F$.

\end{numb}

\begin{numb}\label{5.18}
Now fix a 2-simplex $\Delta\subset X$. We will define $$\psi|_{\Delta}: I_{k}\rightarrow PM_{\bullet}(C(\Delta))P$$ based on the previous definition of $$\psi|_{\partial\Delta}:I_{k}\rightarrow PM_{\bullet}C(\partial\Delta)P.$$
Again, write $$\phi|_{\Delta}=diag_{0\leq i\leq l(\Delta)}(\phi_{i}),$$
where $$\phi_{i}=\phi|_{\Delta}^{[t_{i}(\Delta),s_{i}(\Delta)]}=I_{k}|_{[t_{i},s_{i}]}\rightarrow P_{i}M_{\bullet}(C(\Delta))P_{i}$$ and $P_{i}$ are projections defined on $\Delta$ with $$\sum\limits_{i=0}\limits^{l(\Delta)}P_{i}(x)=P(x),~~~~ \forall x\in\Delta.$$

For each face $\Delta'\subset\partial\Delta$, we know that the decomposition $$Sp\phi|_{\Delta'}=\bigcup\limits_{j=0}^{l_{\Delta'}}T_{j}(\Delta')\bigcap Sp\phi|_{\Delta'}=\bigcup\limits_{j=0}^{l_{\Delta'}}[t_{j}(\Delta'),s_{j}(\Delta')]\bigcap Sp\phi|_{\Delta'}$$
is finer than the decomposition
$$Sp\phi|_{\Delta'}=\bigcup\limits_{j=0}^{l_{\Delta}}T_{j}(\Delta)\bigcap Sp\phi|_{\Delta'}=\bigcup\limits_{j=0}^{l_{\Delta}}[t_{j}(\Delta),s_{j}(\Delta)]\bigcap Sp\phi|_{\Delta'}.$$
Consequently, 
$$[0,s_{0}(\Delta')]\subset[0,s_{0}(\Delta)]~~~~~\mbox{ and}~~~~~[t_{l(\Delta')},1]\subset[t_{l(\Delta)},1].$$

Note that when we define $\psi|_{\Delta'}$ by modifying $\phi|_{\Delta'}$,  we only modify the parts of
$\phi|_{\Delta'}^{[0,s_{0}(\Delta')]}$ and
$\phi|_{\Delta'}^{[t_{l(\Delta')},1]}$ --- that is, 
$$\phi|_{\Delta'}^{[s_{0}(\Delta')+\delta,t_{l(\Delta')}{(\Delta')}-\delta]}=\psi|_{\Delta'}^{[s_{0}(\Delta')+\delta,t_{l(\Delta')}{(\Delta')}-\delta]},$$
where $\delta\in(0,\frac{\eta}{12n})$. Hence
$$\phi|_{\Delta'}^{[t_{1}(\Delta),s_{l(\Delta)-1}(\Delta)]}=\psi|_{\Delta'}^{[t_{1}(\Delta),s_{l(\Delta)-1}(\Delta)]}$$
since$$t_{1}(\Delta)>s_{0}(\Delta)+\frac{\eta}{12n}\geq s_{0}(\Delta')+\delta$$and $$s_{l(\Delta)-1}<t_{l(\Delta)}(\Delta)-\frac{\eta}{12n}<t_{l(\Delta')}(\Delta')-\delta.$$
Because $\Delta'\subset\partial\Delta$ is an arbitrary face, we have
$$\phi|_{\partial\Delta}^{[t_{1}(\Delta),s_{l(\Delta)-1}(\Delta)]}=\psi|_{\partial\Delta}^{[t_{1}(\Delta),s_{l(\Delta)-1}(\Delta)]}.$$

Therefore similar to  what we did on 1-simplexes,  define $$\psi|_{\Delta}^{[t_{j}(\Delta),s_{j}(\Delta)]}=\phi|_{\Delta}^{[t_{j}(\Delta),s_{j}(\Delta)]}$$
for $j\in\{1,2,\cdots,l(\Delta)-1\}$. Then we only need to modify $\phi|_{\Delta}^{[0,s_{0}(\Delta)]}=\phi_{0}$ and $\phi|_{\Delta}^{[t_{l(\Delta)},1]}=\phi_{l}$. We will only do it for $\phi_{0}$.
\end{numb}

\begin{numb}\label{5.19} 
We have the definition of unital homomorphism $$\psi_{0}|_{\partial\Delta}:I_{k}|_{[0,s_{0}]}\rightarrow P_{0}M_{\bullet}(C(\partial\Delta))P_{0}$$
such that $$\sharp(Sp\widetilde{\psi}_{0}|_{x}\cap\{0\})=k'\in\{1,2,\cdots,k\}$$ for any $x\in\partial\Delta$, where  $\Delta$ is a 2-simplex and $\widetilde{\psi}_{0}$
is defined as the composition $$C[0,s_{0}]\hookrightarrow I_{k}|_{(0,s_{0}]}\xrightarrow{\psi_{0}}P_{0}M_{\bullet}(C(\partial\Delta))P_{0}.$$
We need to extend it to a homomorphism $$\psi_{0}|_{\Delta}:I_{k}|_{[0,s_{0}]}\rightarrow P_{0}M_{\bullet}(C(\Delta))P_{0}$$
such that $\sharp(Sp\widetilde{\psi}_{0}|_{\Delta}\cap\{0\})=k'$ for all $x\in\Delta$. Once this extension is obtained, as in \ref{5.17}, we can use the claim in \ref{5.17} to prove that $\phi|_{\Delta}^{[0,s_{0}]}$ and $\psi|_{\Delta}^{[0,s_{0}]}$ are approximately equal to within $\frac{\varepsilon}{3}$ for all $f\in F$. 
(Note that in the argument of \ref{5.17}, the estimation is true which do not depend on the choice of the extension. It only uses $|s_{0}-0|<\eta/2<\eta$, and $\|f(t)-f(t')\|<\frac{\varepsilon}{6}$ whenever $|t-t'|<\eta.$)

There is a $W\in U(M_{\bullet}(C(\Delta)))$ such that $$P_{0}(x)=W(x)\left(
                                                                              \begin{array}{cc}
                                                                                \textbf{1}_{rank(P_{0})} & 0 \\
                                                                                0 & 0 \\
                                                                              \end{array}
                                                                            \right)
W^{\ast}(x)$$

for all  $x\in \Delta$. Again, if we can extend $$(AdW\circ\psi_{0})|_{\partial\Delta}:I_{k}|_{[0,s_{0}]}\rightarrow \left(
                                                                                                     \begin{array}{cc}
                                                                                                       \textbf{1}_{rank(P_{0})} & 0 \\
                                                                                                       0 & 0 \\
                                                                                                     \end{array}
                                                                                                   \right)
M_{\bullet}(C(\Delta))\left(
                                                                                                     \begin{array}{cc}
                                                                                                       \textbf{1}_{rank(P_{0})} & 0 \\
                                                                                                       0 & 0 \\
                                                                                                     \end{array}
                                                                                                   \right),$$
to$$\alpha|_{\Delta}: I_{k}|_{[0,s_{0}]}\rightarrow \left(
                                                                                                     \begin{array}{cc}
                                                                                                       \textbf{1}_{rank(P_{0})} & 0 \\
                                                                                                       0 & 0 \\
                                                                                                     \end{array}
                                                                                                   \right)
M_{\bullet}(C(\Delta))\left(
                                                                                                     \begin{array}{cc}
                                                                                                      \textbf{1}_{rank(P_{0})} & 0 \\
                                                                                                       0 & 0 \\
                                                                                                     \end{array}
                                                                                                   \right),$$
then we can set $\psi_{0}|_{\Delta}=AdW^*\circ\alpha|_{\Delta}$ to obtain our extension. But $(AdW\circ\psi_{0})|_{\partial\Delta}$ (or $\alpha|_{\Delta}$) should be regarded as a homomorphism from $I_{k}|_{[0,s_{0}]}$ to $M_{rank(P_{0})}(C(\partial\Delta))$ (or to $M_{rank(P_{0})}(C(\Delta))$. Hence the construction of $\psi_{0}|_{\Delta}$ follows from the following lemma.
\end{numb}

\begin{lemma}\label{Ext two}  
 Let $\beta: I_{k}|_{[0,s_{0}]}\rightarrow M_{n'}(C(S^{1}))$ be a unital homomorphism such that for any $x\in S^{1}$, $$\sharp(Sp(\beta\circ\imath)_{x}\cap\{0\})=k'\in\{1,2,\cdots,k\}$$ for some fixed  $k'$ (not depending on $x$), where $\imath: C[0,s_{0}]\rightarrow I_{k}|_{[0,s_{0}]}$. Then there is a homomorphism $$\overline{\beta}: I_{k}|_{[0,s_{0}]}\rightarrow M_{n'}(C(D)),$$
(where $D$ is the disk with boundary $S^{1}$)
such that $$\sharp Sp(\overline{\beta}\circ\imath)_{x}\bigcap\{0\}=k'$$ for all $x\in D$ and $\pi\circ\overline{\beta}=\beta$, where $$\pi:M_{n'}(C(D))\rightarrow M_{n'}(C(S^{1}))$$ is the restriction.

\noindent\textbf{Proof.}
Let $h(t)=t\cdot\textbf{1}_{k}$ be the function in the center of $I_{k}|_{[0,s_{0}]}$. Then $\beta(h)$ is a self adjoint element in $M_{n'}(C(S^{1}))$. For each $z\in S^{1}$, write the eigenvalue of $\beta(h)(z)$ in increasing order$$0=\lambda_{1}(z)\leq\lambda_{2}(z)\leq\cdots\leq\lambda_{n'}(z)\leq s_{0}.$$
Then $\lambda_{1},\lambda_{2},\cdots,\lambda_{n'}$ are continuous functions from $S^{1}$ to $[0,s_{0}]$. From the assumption, we know that
$\lambda_{1}(z)=\lambda_{2}(z)=\cdots=\lambda_{k'}(z)=0$ and for all $j>k'$, $\lambda_{j}(z)>0$.  (Note that each $\lambda_{j}(j>k')$ repeats some multiple of $k$ times.)   Consequently, there is $\xi\in(0,s_{0}]$ such that $\lambda_{j}(z)\geq\xi$ for all $j>k'$. Hence $\beta$ factors through as $$I_{k}|_{[0,s_{0}]}\rightarrow I_{k}|_{\{0\}}\oplus I_{k}|_{[\xi,s_{0}]}\xrightarrow{diag(\beta_{0},\beta_{1})}M_{n'}(C(S^{1})),$$
where $$\beta_{0}: I_{k}|_{\{0\}}(=\mathbb{C})\rightarrow Q_{0}M_{n'}(C(S^{1}))Q_{0}$$
and $$\beta_{1}: I_{k}|_{[\xi,s_{0}]}(=M_{k}(C[\xi,s_{0}]))\rightarrow Q_{1}M_{n'}(C(S^{1}))Q_{1}$$ with $$Q_{0}+Q_{1}=\textbf{1}_{n'}\in M_{n'}(C(S^{1})).$$ Note that $rank(Q_{0})=k'$, and $rank(Q_1)=n'-k'$, which is a multiple of $k$. Write $rank(Q_{1})=n'-k'=kk''$. There is a unitary $u\in M_{n}(C(S^{1}))$ such that
$$uQ_{0}u^{\ast}=\left(
                  \begin{array}{cc}
                    \textbf{1}_{k'} & 0 \\
                    0 & 0 \\
                  \end{array}
                \right)
~~~~~~\mbox{and}~~~~~~ uQ_{1}u^{\ast}=\left(
     \begin{array}{cc}
       0 & 0 \\
       0 & \textbf{1}_{n'-k'} \\
     \end{array}
   \right).$$
Hence$$Adu^*\circ\beta=diag(\beta_{0}',\beta_{1}')$$with
$$\beta_{0}': I_{k}|_{\{0\}}(=\mathbb{C})\rightarrow M_{k'}(C(S^{1}))$$
and$$\beta_{1}': I_{k}|_{[\xi,s_{0}]}(=M_{k}(C[\xi,s_{0}]))\rightarrow M_{kk''}(C(S^{1})).$$
Evidently,$$\beta_{0}'(c)=\left(
                            \begin{array}{ccc}
                              c & \ & \ \\
                              \ & \ddots & \ \\
                              \ & \ & c \\
                            \end{array}
                          \right)
=c\cdot\textbf{1}_{k'}\in M_{k'}(C(S^{1})),~~~~~\forall c\in\mathbb{C}.$$
For $\beta_{1}'$, there exist  $\beta'':C[\xi,s_{0}]\rightarrow M_{k''}(C(S^{1}))$ and a unitary $V\in M_{kk''}(C(S^{1}))$ such that
\begin{center}
$V\beta_{1}'(f)V^{\ast}=\beta''\otimes id_{k}(f),~~~~~~\forall f\in M_{k}(C[\xi,s_{0}])$.
\end{center}

Let $$W=\left(
         \begin{array}{cc}
          \textbf{ 1}_{k'} & 0 \\
           0 & V \\
         \end{array}
       \right)
\cdot u.$$ Then $$(AdW^*\circ\beta)(f)=\left(
                                            \begin{array}{cccc}
                                              {f}(\underline0) & \ & \ & \ \\
                                              \ & \ddots & \ & \ \\
                                              \ & \ & {f}(\underline0) & \ \\
                                              \ & \ & \ & \beta''\otimes id_{k}(f) \\
                                            \end{array}
                                          \right).$$

Let $m$ be the winding number of the map
\begin{center}
$S^{1}\ni z\longmapsto$ det$(W(z))\in \mathbb{T}\subseteq\mathbb{C}.$
\end{center}

Then $W\in U(M_{n'}(C(S^{1})))$ is homotopic to $W'\in M_{n'}(C(S^{1}))$ defined by \\
$$W'(z)=\left(
                        \begin{array}{ccccc}
                          z^m & \ & \ & \ & \ \\
                          \ & 1 & \ & \ & \ \\
                          \ & \ & 1 & \ & \ \\
                          \ & \ & \ & \ddots & \ \\
                          \ & \ & \ & \ & 1 \\
                        \end{array}
                      \right)
, ~~~\forall z\in S^1=\T.$$
Let $\{w_r\}_{\frac{1}{2}\leq r\leq1}$ be a unitary path in $M_{n'}(C(S^1))$ with 
$$w_{\frac{1}{2}}(z)=W'(z)~~~\mbox{ and }~~~w_1(z)=W(z), ~~~~~\forall z\in S^1.$$ 
Evidently the homomorphism $$\beta'':C[\xi,s_{0}]\rightarrow M_{k''}(C(S^{1}))$$ is homotopic to the homomorphism $$\beta''':C[\xi,s_{0}]\rightarrow M_{k''}(C(S^{1}))$$
defined by $$\beta'''(f)(e^{2\pi i\theta})=f(\xi)\textbf{1}_{k''}$$ --- that is,  $\beta'''(f)(e^{ i\theta})$ is the constant matrix $f(\xi)\textbf{1}_{k''}$ (which does not depend on $\theta$). There is a path $\{\beta_{r}\}_{0\leq r\leq\frac{1}{2}}$ of homomorphisms $$\beta_{r}:C([\xi,s_{0}])\rightarrow M_{k''}(C(S^{1}))$$
such that $\beta_{\frac{1}{2}}=\beta'' $ and $\beta_{0}=\beta'''.$

Finally, regard $D=\{re^{i\theta},0\leq r\leq1\}$, and define $\overline{\beta}:~ I_{k}|_{[0,s_{0}]}\rightarrow M_{n'}(C(D))$ by
$$ \overline{\beta}(f)(re^{i\theta})=
\left\{ \begin{array}
      {r@{\quad \quad}l}
  w_r^{\ast}(e^{i\theta})\left(
               \begin{array}{cc}
               \hspace{-.1in} \left(
        \begin{array}{ccc}
          \hspace{-.1in}{f}(\underline0) & \ & \ \\
          \ & \hspace{-.2in}\ddots & \  \\
          \ & \ &\hspace{-.2in}{f}( \underline0)  \\
        \end{array}
      \hspace{-.08in}\right)_{\!\!\!k'\times k'} &  \\
                  & \hspace{-.4in}(\beta''\otimes id_{k})(f)(e^{i\theta})\\
               \end{array}
             \right)w_r(e^{i\theta}),~ \mbox{if} ~\frac{1}{2}\leq r\leq1
   \\ \left(
               \begin{array}{cc}
                \left(
        \begin{array}{ccc}
          \hspace{-.1in}{f}( \underline0) & \ & \ \\
          \ & \hspace{-.2in}\ddots & \  \\
          \ & \ & \hspace{-.2in}{f}( \underline0)  \\
        \end{array}
      \right)_{\!\!\!k'\times k'} &  \\
                  & \hspace{-.3in}(\beta_{r}\otimes id_{k})(f)(e^{i\theta})\\
               \end{array}
             \right) ,~~~~~~~\mbox{if}~~ 0\leq r\leq\frac{1}{2}.
   &
   \end{array} \right.$$

This homomorphism is as desired.
\end{lemma}

\begin{numb}\textbf{Proof of Theorem \ref{THM}} 
From \ref{5.3}---\ref{Ext two}, we have constructed $$\psi:I_{k}\rightarrow PM_{\bullet}(C(X))P$$ with the property $$\|\phi(f)-\psi(f)\|<\frac{\varepsilon}{3}$$ for all $f\in F$. And importantly,  for each $x\in X$, $\sharp(Sp\widetilde{\psi}|_{x}\cap\{0\})$
is a constant $k'\in\{1,2,\cdots,k\}$ and $\sharp(Sp\widetilde{\psi}|_{x}\cap\{1\})$ is also a constant $k_{1}'\in\{1,2,\cdots,k\}$, where $\widetilde{\psi}$ is the composition $$C[0,1]\hookrightarrow I_{k}\xrightarrow{\psi}PM_{\bullet}(C(X))P.$$

Let $h(t)=t\cdot\textbf{1}_{k}\in I_{k}$ be the canonical function in the center of $I_{k}$. Then $\psi(h)\in PM_{\bullet}(C(X))P$ is a self adjoint element. For each $x\in X$, denote the eigenvalues of
$\psi(h)(x)$ by $$0\leq\lambda_{1}(x)\leq\lambda_{2}(x)\leq\cdots\leq\lambda_{rank(P)}(x)\leq1.$$
Then all $\lambda_{i}(x)$ are continuous functions from $X$ to $[0,1].$ Furthermore,
$$\lambda_{1}(x)=\lambda_{2}(x)=\cdots=\lambda_{k'}(x)=0,$$
$$0<\lambda_{k'+1}(x)\leq\lambda_{k'+2}(x)\leq\cdots\leq\lambda_{rank(P)-k_{1}'}(x)<1,$$
and $$\lambda_{rank(P)-k_{1}'+1}(x)=\lambda_{rank(P)-k_{1}'+2}(x)=\cdots=\lambda_{rank(P)}(x)=1.$$

Let
\begin{center}
$\xi_{1}=\min\limits_{x\in X}\lambda_{k'+1}(x)>0$~~~~~ and~~~~~~ $\xi_{2}=\max\lambda_{rank(P)-k_{1}'}(x)<1.$
\end{center}
Then $$Sp\psi\subset\{0\}\cup[\xi_{1},\xi_{2}]\cup\{1\}.$$ That is,  $\psi$ factors through as
$$I_{k}\rightarrow \mathbb{C}\oplus M_{k}(C[\xi_{1},\xi_{2}])\oplus\mathbb{C}\xrightarrow{diag(\alpha_{0},\psi_{1},\alpha_{1})}PM_{\bullet}(C(X))P,$$ where we identify $I_{k}|_{\{0\}}=\mathbb{C}$
and $I_{k}|_{\{1\}}=\mathbb{C}$.

Let $Q_{0}=\alpha_{0}(1), Q_{1}=\alpha_{1}(1)$ and $P_{1}=\psi_{1}(\textbf{1}_{M_{k}(C([\xi_{1},\xi_{2}]))})$. Finally,  regarding $\psi_{1}$ as
$$M_{k}(C[0,1])\xrightarrow{restriction}M_{k}(C([\xi_{1},\xi_{2}]))\xrightarrow{\psi_{1}}P_{1}M_{\bullet}(C(X))P_{1},$$
we finish the proof of Theorem \ref{THM}.\\
$~~~~~~~~~~~~~~~~~~~~~~~~~~~~~~~~~~~~~~~~~~~~~~~~~~~~~~~~~~~~~~~~~~~~~~~~~~~~~~~~~~~~~~~~~~~~~~~~~~~~~~~~~~~~~~~~~~\Box$

\end{numb}

\begin{numb} From the definition of $\psi$ in the above procedure, for every  $x\in X$, the map $$\psi|_{x}:I_{k}\xrightarrow{\psi}PM_{\bullet}(C(X))P\xrightarrow{evaluate\ at\ x}P(x)M_{\bullet}(\mathbb{C})P(x)$$
is defined when the construction of
$$\psi|_{\Delta}:I_{k}\rightarrow PM_{\bullet}(C(\Delta))P$$
is carried out for the unique simplex $\Delta$ such that $x\in\stackrel{\circ}{\Delta}$ (the interior of $\Delta$). And when we define $\psi|_{\Delta}$ by modifying $\phi|_{\Delta}$, the only modifications are made on the two parts $\phi|_{\Delta}^{[0,s_{0}(\Delta)]}$ and $\phi|_{\Delta}^{[t_{l}(\Delta),1]}$. Consequently,
$$Sp\phi|_{x}\cap(s_{0}(\Delta),t_{l}(\Delta))=Sp\psi|_{x}\cap(s_{0}(\Delta),t_{l}(\Delta))$$
as sets with multiplicity. On the other hand for any simplex $\Delta$, $s_{0}(\Delta)<\frac{\eta}{2}$ and $t_{l(\Delta)}(\Delta)>1-\frac{\eta}{2}$. Hence $$Sp\phi|_{x}\cap[\frac{\eta}{2},1-\frac{\eta}{2}]=Sp\psi|_{x}\cap[\frac{\eta}{2},1-\frac{\eta}{2}].$$ If we further assume that $\phi$ has property $sdp(\eta/4,\delta)$, then $\psi$ has property $sdp(\eta,\delta)$. As a consequence, we can use the decomposition theorem for $$\psi_{1}:M_{k}(C[0,1])\rightarrow P_{1}M_{\bullet}(C(X))P_{1}$$ to study the homomorphisms ~$\phi,\psi:I_{k}\rightarrow PM_{\bullet}(C(X))P.$~
Note that the homomorphisms $f\mapsto {f}(\underline0)Q_{0}$ and $f\mapsto {f}(\underline1)Q_{1}$  factor through the  $C^{\ast}$-algebra $\C$.

\end{numb}

\begin{numb} Lemma \ref{Ext two} is not true for the case $k'=0$. In fact, there exists a unital homomorphism $\alpha: M_{k}(\mathbb{C})\rightarrow M_{k}(C(S^{1})),$
which can not be extended to a homomorphism $\overline{\alpha}: M_{k}(\mathbb{C})\rightarrow M_{k}(C(D)).$ 
Let $\pi_{s_0}: ~I_k|_{[0,s_0]}\to M_k(\C)$ be the map defined by evaluating at the point $s_0$. Then $\beta=\alpha\circ\pi_{s_0}: ~I_k|_{[0,s_0]}\to M_k(C(S^1))$ can not be extended to $\overline{\beta}: ~I_k|_{[0,s_0]}\to M_k(C(D))$ such that  $\sharp Sp(\overline{\beta}\circ\imath)_{x}\bigcap\{0\}=k'=0$ for all $x\in D$, where $\imath$ is the canonical map from $M_k(\mathbb{C})$ to $I_k|_{[\zeta,s_0]}$ for some $0<\zeta<s_0$. 
\end{numb}

\section{Decomposition Theorem II}

 Our next task is to study the possible decomposition of $\phi: C(X)\rightarrow M_{l}(I_{k_{2}})$ for $X$ being $[0,1],~S^{1}$ or $\mathrm{T}_{\uppercase\expandafter{\romannumeral2},k}$. The cases of $[0,1],~S^{1}$ are more or less known (see \cite{Ell1} and \cite{EGJS}). Let us assume $X$ is a 2-dimensional connected simplicial complex.

The following  lemma is essentially due to H. Su (See \cite{Su}). The case of $X=\noindent\textbf{graph}$ was stated in \cite{Li1}.

\begin{lemma}\label{5.24} For any connected simplicial complex $X$, a finite set $F\subset C(X)$ which generates $C(X)$, $\eta>0$ and a positive interger $n>0$, there is a $\delta>0$, such that for any two unital homomorphisms $\phi,\psi: C(X)\rightarrow M_{n}(\mathbb{C}),$ if $\|\phi(f)-\psi(f)\|<\delta$~ for all $f\in F,$
~then $Sp(\phi)$ and $Sp(\psi)$ can be paired within $\eta$.
\end{lemma}
This is a consequence of Lemma 2.2 and Lemma 2.3 of \cite{Su}; also see the argument 2.1.3 in \cite{Li1}. For the case of graphs, it was stated in 2.1.9 of \cite{Li1}.

\begin{lemma}\label{su} For any connected simplicial complex $X$, a finite generating set  $F\subset C(X)$, ~$\varepsilon>0$ and positive integer $n>0$, there is $\delta>0$ with the following property: If $x_{1},x_{2},\cdots,x_{n}\in X$ are $n$ points (possibly repeating), $u,v\in M_{n}(\mathbb{C})$ are two unitaries such that
$$\left\|u\left(
       \begin{array}{cccc}
         f(x_{1}) & \ & \ & \ \\
         \ & f(x_{2}) & \ & \ \\
         \ & \ & \hspace{-.1in}\ddots & \ \\
         \ & \ & \ & \hspace{-.1in}f(x_{n}) \\
       \end{array}
     \right)
u^{\ast}-v\left(
       \begin{array}{cccc}
         f(x_{1}) & \ & \ & \ \\
         \ & f(x_{2}) & \ & \ \\
         \ & \ & \hspace{-.1in}\ddots & \ \\
         \ & \ & \ & \hspace{-.1in}f(x_{n}) \\
       \end{array}
     \right)v^{\ast}\right\|<\delta,$$
for all $f\in F$, then there is a path of unitaries $u_{t}\in M_{n}(\mathbb{C})$ connecting $u$ and $v$ (i.e, $u_{0}=u, u_{1}=v$) with the property that
$$\left\|u_{t}\left(
       \begin{array}{cccc}
         f(x_{1}) & \ & \ & \ \\
         \ & f(x_{2}) & \ & \ \\
         \ & \ & \hspace{-.1in}\ddots & \ \\
         \ & \ & \ & \hspace{-.1in}f(x_{n}) \\
       \end{array}
     \right)
u_{t}^{\ast}-u_{t'}\left(
       \begin{array}{cccc}
         f(x_{1}) & \ & \ & \ \\
         \ & f(x_{2}) & \ & \ \\
         \ & \ & \hspace{-.1in}\ddots & \ \\
         \ & \ & \ & \hspace{-.1in}f(x_{n}) \\
       \end{array}
     \right)u_{t'}^{\ast}\right\|<\varepsilon$$
for all $f\in F$ and $t,t'\in [0,1]$ (of course $\delta$ depends on both $\varepsilon$ and $n$).
\end{lemma}
This was proved in step 2 and step 3 of the proof of Theorem 3.1 of \cite{Su}. 


The following lemma reduces the study of $\phi:C(X)\rightarrow M_{l}(I_{k})$ to the study of homomorphism $\phi_{1}:C(\Gamma)\rightarrow M_{l}(I_{k})$,
where $\Gamma\subset X$ is 1-skeleton of $X$ under a certain simplicial decomposition. Since $\Gamma$ is a graph, then we will apply the technique in \cite{Li1} and \cite{Li2} to obtain the decomposition of $\phi_{1}$.

\begin{lemma}\label{5.26} Let $X$ be a 2-dimensional  simplicial complex. For any $F\!\subset\! C(X),~\varepsilon\!>\!0,~\eta\!>\!0$,
and any unital homomorphism $\phi:C(X)\rightarrow M_{l}(I_{k})$,
there is a simpicial decomposition of $X$ with 1-skeleton $X^{(1)}=\Gamma$ and  a homomorphism
$\phi_{1}:C(\Gamma)\rightarrow M_{l}(I_{k})$ such that:

1. $\|\phi(f)-\phi_{1}\circ\pi(f)\|<\varepsilon, $ where $\pi: C(X)\rightarrow C(\Gamma)$ is given by $\pi(f)=f|_{\Gamma}$; 

2. For any $t\in[0,1]$, $Sp\phi|_{t}$ and $Sp(\phi_{1}\circ\pi)_{t}$ can be paired within $\eta$.
\end{lemma}

\begin{proof} By Lemma \ref{5.24}, we only need to prove that there exists a homomorphism $\phi_{1}$ to satisfy condition (1). Without loss of generality, we  assume that $\mathrm{F}$ generates $C(X)$. By \ref{su}, there is an $\varepsilon'>0$ such that for any $x_{1},x_{2},\cdots,x_{kl}\in X$ and unitaries $u,v\in M_{kl}(\mathbb{C}),$ if
$$\left \|u\left(
       \begin{array}{cccc}
         f(x_{1}) & \ & \ & \ \\
         \ & f(x_{2}) & \ & \ \\
         \ & \ & \hspace{-.1in}\ddots & \ \\
         \ & \ & \ & f(x_{kl}) \\
       \end{array}
     \right)
u^{\ast}-v\left(
       \begin{array}{cccc}
         f(x_{1}) & \ & \ & \ \\
         \ & f(x_{2}) & \ & \ \\
         \ & \ & \hspace{-.1in}\ddots & \ \\
         \ & \ & \ & f(x_{kl}) \\
       \end{array}
     \right)v^{\ast} \right \|<\varepsilon',$$
then there is a continuous path $u_{t}$ with $u_{0}=u, u_{1}=v$ satisfying that
$$\left \|u_{t}\left(
       \begin{array}{cccc}
         f(x_{1}) & \ & \ & \ \\
         \ & f(x_{2}) & \ & \ \\
         \ & \ & \hspace{-.1in}\ddots & \ \\
         \ & \ & \ & \hspace{-.1in}f(x_{kl}) \\
       \end{array}
     \right)
u_{t}^{\ast}-u_{t'}\left(
       \begin{array}{cccc}
         f(x_{1}) & \ & \ & \ \\
         \ & f(x_{2}) & \ & \ \\
         \ & \ & \hspace{-.1in}\ddots & \ \\
         \ & \ & \ & \hspace{-.1in}f(x_{kl}) \\
       \end{array}
     \right)u_{t'}^{\ast} \right \|<\frac{\varepsilon}{3}.$$

Recall for the simplicial complex, a continuous path $\{x(t)\}_{0\leq t\leq1}$ is called piecewise linear if there are a sequence of points $$0=t_{0}<t_{1}<\cdots<t_{n}=1$$
such that $\{x(t)\}_{t_{i}\leq t\leq t_{i+1}}$ fall in the same simplex of $X$ and are linear there. Note that the property of piecewise linear is preserved under any subdivision of the simplicial complex. For the simplicial complex $X$, we endow the standard metric on $X$, briefly described as below (see  \cite[1.4.1]{G5} for detail). Identify each n-simplex with an n-simplex in $\mathbb{R}^{n}$ whose edges are of length 1, preserving affine structure of the simplexes. Such identifications give rise to a unique metric on the simplex $\Delta$. For any two points $x,y\in X$, $d(x,y)$ is defined to be the length of the shortest path connecting $x$ and $y$. (The length is measured in individual simplex, by breaking the path into small pieces). With this metric, if $x_{0},x_{1}\in X$ with $d(x_{0},x_{1})=d$, then there is a piecewise linear path $x(t)$ with length $d$ such that $x(0)=x_{0},~x(1)=x_{1}$. Furthermore, $d(x(t),x(t)')\leq d$ for all $t,t'\in[0,1]$. In fact, we can choose $x(t)$, such that $$d(x(t),x(t'))=|t'-t|\cdot d.$$
There is an $\eta'<\frac{\eta}{4}$ such that the following is true: For any $x,x'\in X$ with $d(x,x')<2\eta'$, $$|f(x)-f(x)'|<\frac{\varepsilon'}{3}.$$

Let $\delta>0$, such that if $|t-t'|\leq\delta$, then $$\|\phi(f)(t)-\phi(f)(t')\|<\frac{\varepsilon'}{3},~ \forall f\in F,$$ and $Sp\phi_{t}$ and  $Sp\phi_{t'}$ can be paired  within $\eta'$.

Dividing the interval $[0,1]$ into pieces $0=t_{0}<t_{1}<t_{2}<\cdots<t_{\bullet}=1$, with $|t_{i+1}-t_{i}|<\delta$.
We first define $\psi: C(X)\rightarrow M_{l}(I_{k})$ such that $\psi$ is close to $\phi$ on $F$ to within $\frac{\varepsilon}{3}$, $Sp\phi_{t}$ and $Sp\psi_{t}$ can be paired within $\eta'$, and with extra property that on each interval $[t_{i},t_{i+1}]$; $Sp\psi_{t}=\{\alpha_{1}(t),\alpha_{2}(t),\cdots,\alpha_{lk}(t)\}$ with all $\alpha_{j}:[t_{i},t_{i+1}]\rightarrow X$ being piecewise linear. 

Set $\psi|_{\{t_{i}\}}=\phi|_{\{t_{i}\}}$, for each $t_{i}$ ($i=0,1,2,\cdots,\bullet$)--- that is,
$$\psi(f)(t_{i})=\phi(f)(t_{i})~~~~~ \mbox{ for}~ i=0,1,2,\cdots,\bullet.$$
And we will define $\psi|_{\{t\}}$ for $t\in(t_{i},t_{i+1})$ by interpolating  the definitions between $\psi|_{\{t_{i}\}}$ and $\psi|_{\{t_{i}+1\}}$. (Note that we do not change the definitions of $\phi|_{\{0\}}$ and $\phi|_{\{1\}}$, hence $\psi$ is a homomorphism into $M_{l}(I_{k})$ instead of $M_{lk}(C[0,1])$.)

Let $$Sp\psi|_{\{t_{i}\}}=\{\alpha_{1},\alpha_{2},\cdots,\alpha_{lk}\}\subset X$$
$$Sp\psi|_{\{t_{i+1}\}}=\{\beta_{1},\beta_{2},\cdots,\beta_{lk}\}\subset X.$$

Since $Sp\psi|_{\{t_{i}\}}$ and $Sp\psi|_{\{t_{i+1}\}}$ can be paired within $\eta'$, we can assume $dist(\alpha_{i},\beta_{i})<\eta'$. There exist two unitaries $u,v\in M_{lk}(\mathbb{C})$ such that
$$\psi(f)(t_{i})=u\left(
                    \begin{array}{ccc}
                      \hspace{-.05in}f(\alpha_{1}) & \ & \ \\
                      \ & \hspace{-.1in}\ddots & \ \\
                      \ & \ & \hspace{-.1in}f(\alpha_{kl}) \\
                    \end{array}
                  \right)
u^{\ast}~~\mbox{and}~~~
\psi(f)(t_{i+1})=v\left(
                    \begin{array}{ccc}
                      \hspace{-.05in}f(\beta_{1}) & \ & \ \\
                      \ & \hspace{-.1in}\ddots & \ \\
                      \ & \ & \hspace{-.1in}f(\beta_{kl}) \\
                    \end{array}
                  \right)
v^{\ast}.$$
Note that $\|f(\alpha_{j})-f(\beta_{j})\|<\frac{\varepsilon'}{3}$ for each $j$, we have
$$\left \|v\left(
                    \begin{array}{ccc}
                      f(\alpha_{1}) & \ & \ \\
                      \ & \ddots & \ \\
                      \ & \ & f(\alpha_{kl}) \\
                    \end{array}
                  \right)
v^{\ast}-v\left(
                    \begin{array}{ccc}
                      f(\beta_{1}) & \ & \ \\
                      \ & \ddots & \ \\
                      \ & \ & f(\beta_{kl}) \\
                    \end{array}
                  \right)
v^{\ast}\right \|<\frac{\varepsilon'}{3}.$$
Combining with $\|\psi(f)(t_{i})-\psi(f)(t_{i+1})\|<\frac{\varepsilon'}{3}$, we get
$$\left \|u\left(
                    \begin{array}{ccc}
                      f(\alpha_{1}) & \ & \ \\
                      \ & \ddots & \ \\
                      \ & \ & f(\alpha_{kl}) \\
                    \end{array}
                  \right)
u^{\ast}-v\left(
                    \begin{array}{ccc}
                      f(\alpha_{1}) & \ & \ \\
                      \ & \ddots & \ \\
                      \ & \ & f(\alpha_{kl}) \\
                    \end{array}
                  \right)
v^{\ast}\right \|<\frac{2\varepsilon'}{3}.$$
Since $\varepsilon'$ is the number $\delta$ in Lemma \ref{su} for $\frac{\varepsilon}{3}$, applying Lemma \ref{su}, there is a unitary path $u(t)$, ${t_{i}\leq t\leq\frac{t_{i}+t_{i+1}}{2}}$~
with~~ $u(t_{i})=u,~~~
u(\frac{t_{i}+t_{i+1}}{2})=v$~ such that
$$\left \|u(t)\left(
                    \begin{array}{ccc}
                      f(\alpha_{1}) & \ & \ \\
                      \ & \hspace{-.1in}\ddots & \ \\
                      \ & \ & f(\alpha_{kl}) \\
                    \end{array}
                  \right)
u^{\ast}(t)-u(t')\left(
                    \begin{array}{ccc}
                      f(\alpha_{1}) & \ & \ \\
                      \ & \hspace{-.1in}\ddots & \ \\
                      \ & \ & f(\alpha_{kl}) \\
                    \end{array}
                  \right)
u^{\ast}(t')\right \|<\frac{\varepsilon}{3}$$
for all $t,t'\in[t_{i},\frac{t_{i}+t_{i+1}}{2}]$.

There are piecewise linear paths $r_{i}(t)$ with $r_{i}(\frac{t_{i}+t_{i+1}}{2})=\alpha_{i}$ and $ r_{i}(t_{i+1})=\beta_{i}$ such that $$d(r_{i}(t),r_{i}(t'))\leq dist(\alpha_{i},\beta_{i})<\eta'.$$ Define $\psi(f)$ as follows: 
For $t\in[t_{i},\frac{t_{i}+t_{i+1}}{2}]$,
$$\psi(f)(t)=u(t)\left(
                    \begin{array}{ccc}
                      f(\alpha_{1}) & \ & \ \\
                      \ & \ddots & \ \\
                      \ & \ & f(\alpha_{kl}) \\
                    \end{array}
                  \right)
u^{\ast}(t);$$
for $t\in[\frac{t_{i}+t_{i+1}}{2},t_{i+1}]$,
$$\psi(f)(t)=v\left(
                \begin{array}{cccc}
                  f(r_{1}(t)) & \ & \ & \  \\
                  \ & f(r_{2}(t)) & \ & \ \\
                  \ & \ & \ddots & \  \\
                  \ & \ & \ & f(r_{kl}(t)) \\
                \end{array}
              \right)
v^{\ast}.$$
Then $\{Sp\psi_{t}, t\in[t_{i},t_{i+1}]\}$ is a collection of $kl$ piecewise linear maps from $[t_{i},t_{i+1}]$ to $X$. (Note that for $t\in[t_{i},\frac{t_{i}+t_{i+1}}{2}]$, we use constant maps which are linear.)

Now  subdivid the simplicial complex $X$ so that  each simplex of the subdivision has diameter at most $\eta'$,  and so that all the points in
$Sp\phi|_{\{0\}}=Sp\psi|_{\{0\}}$ and $Sp\phi|_{\{1\}}=Sp\psi|_{\{1\}}$ are vertices. With this simplicial decomposition  we have, 
$Sp\psi\cap\Delta\subsetneqq\Delta$, for every $2$-simplex $\Delta$. This is true because $Sp\psi|_{[t_{i},t_{i+1}]}$ is the union of the collection of images of $kl$ piecewise linear maps from $[t_{i},t_{i+1}]$ to $X$, and a finite union of line segments must be  1-dimensional. Hence for each simplex $\Delta$
of dimension 2, we can choose a point $x_{\Delta}\in\stackrel{\circ}{\Delta}$, such that $x_{\Delta}\not\in Sp \psi$.

There is a $\sigma>0$ such that $Sp\psi$ has no intersection with
 $\overline{B_{\sigma}(x_{\Delta})}=\{x\in X,dist(x,x_{\Delta})\leq\sigma\}$
for all $\Delta$. 
Let $Y=X\backslash\big(\cup\{B_{\sigma}(x_{\Delta})~| \Delta\mbox{ is 2-simplex}\}\big)$. Then $Sp\psi\subset Y$. That is, $\psi$ factors through $C(Y)$ as $$\psi:~~C(X)\xrightarrow{restriction}C(Y)\xrightarrow{\psi_{1}}M_{l}(I_{k}).$$
Let $\alpha:Y\rightarrow X^{(1)}$ be the standard retraction defined as a map sending $\Delta\backslash\{x_{\Delta}\}$ to $\partial\Delta$ for each simplex $\Delta$. Then $d(x,\alpha(x))<\eta'$.
Let $\phi_{1}:C(X^{(1)})\rightarrow M_{l}(I_{k})$ be defined by $$\psi_{1}\circ\alpha^{\ast}:C(X^{(1)})\xrightarrow{\alpha^{\ast}}C(Y)\xrightarrow{\psi_{1}}M_{l}(I_{k}).$$
Evidently $\phi_{1}$ is as desired.  
\end{proof}

\begin{corollary}\label{5.27} Suppose that $\phi:C(X)\rightarrow M_{l}(I_{k})$ is a unital homomorphism. For any finite set $F\subset C(X)$, $\varepsilon>0$, and $\eta>0$, there is a unital homomorphism $$\psi:C(X)\rightarrow M_{l}(I_{k})$$
such that \\
(1) $\phi(f)(0)=\psi(f)(0), \phi(f)(1)=\psi(f)(1)$ for all $f\in C(X)$;\\
(2) $\|\phi(f)-\psi(f)\|<\varepsilon$ for all $f\in F$;\\
(3) $Sp\phi_{t}$ and $Sp\psi_{t}$ can be paired to within $\eta$;\\
(4) For each $t\in(0,1)$, the maximal multiplicity of $Sp\psi_{t}$ is one --- that is, $\psi|_{\{t\}}$ has distinct spectra.
\end{corollary}

\begin{proof} Applying Lemma \ref{5.26}, we reduce the case of $C(X)$ to the case of $C(X^{(1)})$, where $X^{(1)}$ is a 1-dimensional simplicial complex. 
The corollary of this case is almost the same as the special case  of \cite[Theorem 2.1.6]{Li1} (where we let $Y=[0,1]$).
Note that from the proof of Theorem 2.1.6 in \cite{Li1}, if we do not require the homomorphism $\psi$ to have distinct spectrum at the end points 0 and 1, then we do not need to modify the original homomorphism $\phi$ at these two end points. The proof goes the same way as the proof there with some small modifications. 
We briefly describe them as below. One divides the interval $Y=[0,1]$ into small pieces $[0,1]=\cup_{i=0}^{m-1}[y^i,y^{i+1}]$ with $y^0=0<y^1<y^2\cdots<y^m=1$, as in the proof of \cite[Theorem 2.1.6]{Li1}. Define $\psi|_{y^i}$ with $1\leq i\leq m-1$, by slightly modifying  $\phi|_{y^i}$ so that $\psi|_{y^i}$ has distinct spectra; but define $\psi|_0=\phi|_0$ and $\psi|_1=\phi|_1$ (no modification are made at the ending points).  Therefore, in our case, $\psi|_0$ and $\psi|_1$ do not have distinct spectra---this is the only difference from \cite[Theorem 2.1.6]{Li1}.  For all intervals $[y^i,y^{i+1}]$ with $1\leq i\leq m-2$, the constructions of  $\psi|_{[y^i,y^{i+1}]}$ are the same as in the proof of \cite[Theorem 2.1.6]{Li1}. For the constructions of  $\psi|_{[0,y^1]}$ and $\psi|_{[y^{m-1}, 1]}$, we need to modify \cite[Lemma 2.1.1]{Li1} and \cite[Lemma 2.1.2]{Li2} accordingly, in an obvious way, and then apply these modifications. For example,  \cite[Lemma 2.1.1]{Li1} should be modified to the following case: among two $l$-element sets $X^{0}=\{x_{1}^{0},x_{2}^{0},\cdots,x_{l}^{0}\}$ and
$X^{1}=\{x_{1}^{1},x_{2}^{1},\cdots,x_{l}^{1}\}$ --- only one of them is distinct.  That is, the following statement is true with the same proof:

Let $X=X_{1}\vee X_{2}\vee\cdots\vee X_{k}$ be a bunch of $k$ intervals $X_{i}=[0,1]~(1\leq i\leq k)$ and $Y=[0,1]$.  Suppose that $$X^{0}=\{x_{1}^{0},x_{2}^{0},\cdots,x_{l}^{0}\}\subset X ~~~\mbox{ and}~~~~
X^{1}=\{x_{1}^{1},x_{2}^{1},\cdots,x_{l}^{1}\}\subset X$$ with $x_{i}^{1}\neq x_{j}^{1}$ if $i\neq j$.
Then there are $l$ continuous functions $f_{1},f_{2},\cdots,f_{l}:Y\rightarrow X$ such that \\
(1) as sets with multiplicity, we have
$$\{f_{1}(0),f_{2}(0),\cdots,f_{l}(0)\}=X^{0}, ~~~\mbox{and} ~~~\{f_{1}(1),f_{2}(1),\cdots,f_{l}(1)\}=X^{1},$$
(2) for each $t\in(0,1]\subset Y$ and $i\neq j$, we have $$f_{i}(t)\neq f_{j}(t).$$
\end{proof}

\begin{remark}\label{5.28} In Corollary \ref{5.27}, we can further assume that $Sp\psi|_{\{0\}}$ and $Sp\psi|_{\{1\}}$ have eigenvalue multiplicity  just $k$ as homomorphisms from $C(X)$ to $M_{lk}(C[0,1])$, or equivalently,  both maps
$$
C(X)\xrightarrow{\psi}M_{l}(I_{k})\xrightarrow{evaluate\ at\ 0}M_{l}(\mathbb{C})
~~~\mbox{and}~~~
C(X)\xrightarrow{\psi}M_{l}(I_{k})\xrightarrow{evaluate\ at\ 1}M_{l}(\mathbb{C})
$$
have distinct spectrum. To do this, we first extend the definition of the original $\phi$ to a slightly larger interval $[-\delta,1+\delta]$ as below.

Find $u\in M_{l}(\mathbb{C})$ and $x_{1},x_{2},\cdots,x_{l}\in X$ such that
$$\phi(f)(0)=u\left(
                \begin{array}{cccc}
                  f(x_{1}) & \ & \ & \  \\
                  \ & f(x_{2}) & \ & \ \\
                  \ & \ & \ddots & \  \\
                  \ & \ & \ & f(x_{l}) \\
                \end{array}
              \right)u^{\ast}\otimes \textbf{1}_{k}.$$
Since $X$ is path connected and $X\neq\{pt\}$, there are functions $\alpha_{i}:[-\delta,0]\rightarrow X$ such that $\{\alpha_{i}(-\delta)\}_{i=1}^{l}$
is a set of distinct $l$ points, $\alpha_{i}(0)=x_{i}$, and $dist(\alpha_{i}(t),\alpha_{i}(0))$ are as  small as we want.  Define
\begin{center}
$\phi(f)(t)=u\left(
                \begin{array}{ccc}
                  f(\alpha_{1}(t)) & \ & \   \\
                  \ & \ddots & \ \\
                  \ & \ & f(\alpha_{l}(t))  \\
                \end{array}
              \right)u^{\ast}\otimes\textbf{1}_{k}, ~~~~~\mbox{for } t\in[-\delta,0].$
\end{center}
Similarly, we can define $\phi(f)(t)$ for $t\in[1,1+\delta]$, so that $\phi|_{1+\delta}$ as a homomorphism from $C(X)$ to $M_{kl}(\mathbb{C})$ has multiplicity exactly $k$ and $\phi(f)(1+\delta)\in M_{l}(\mathbb{C})\otimes\textbf{1}_{k}$. One can reparemetrize $[-\delta,1+\delta]$ to $[0,1]$ so that $\phi|_{0}$ and $\phi|_{1}$ as homomorphisms from $C(X)$ to $M_{kl}(\mathbb{C})$
have  multiplicity exactly $k$. Then we apply the corollary to perturb $\phi$ to $\psi$ without changing the definition at the end points.
\end{remark}

\begin{remark}\label{5.29} The same argument can be used to prove the following result. Let $X\neq\{pt\}$ be a connected finite simplicial complex of any
dimension. Let $Y$ be a 1-dimensional simplicial complex.  Then any homomorphism $\phi:C(X)\rightarrow M_{n}(C(Y))$ can be approximated arbitrarily well by a homomorphism $\psi$ with distinct spectrum. This is a strengthened form of \cite[Theorem 2.1]{G5} for the case $dim(Y)=1$.

The following Theorem for $X=\noindent\textbf{gragh}$, is a slight modification of \cite[Theorem 2.7]{Li3}. 
\end{remark}

\begin{theorem}\label{5.30} Let $X$ be a connected simplicial complex of dimension at most 2, and $G\subset C(X)$ be a finite set which generates $C(X)$. For any $\varepsilon>0$, there is an $\eta>0$ such that the following statement  is true.

Suppose that $\phi:C(X)\rightarrow M_{l_{1}l_{2}+r}(I_{k})$ is a unital homomorphism satisfying the following condition:
There are $l_{1}$ continuous maps $$a_{1},a_{2},\cdots,a_{l_{1}}:[0,1](=Sp(I_{k}))\rightarrow X$$
such that for every $y\in[0,1]$, $Sp\phi_{y}$ (considered as a homomorphism from $C(X)$ to $ M_{(l_{1}l_{2}+r)k}(C[0,1])$) and $\Theta(y)$ can be paired within $\eta$, where $$\Theta(y)=\{a_{1}(y)^{\sim l_{2}k},a_{2}(y)^{\sim l_{2}k},\cdots,a_{l_{1}-1}(y)^{\sim l_{2}k},a_{l_{1}}(y)^{\sim (l_{2}+r)k}\}.$$
It follows that there are $l_{1}$ mutually orthogonal projections $p_{1},p_{2},$ $\cdots,$ $p_{l_{1}}\in$ $M_{l_{1}l_{2}+r}(I_{k})$ such that\\
(i) for all $g\in G$ and $y\in Y$
$$\|\phi(g)(y)-p_{0}\phi(g)(y)p_{0}\oplus\sum_{k=1}^{l_{1}}g(a_{k}(y))p_{k}\|<\varepsilon,$$ where 
 $p_{0}=1-\sum_{i=1}^{l_{1}}p_{i}$;\\
(ii) $rank(p_{i})=(l_{2}-3)k$ for $1\leq i<l_1$,
$rank(p_{l_{1}})=(l_{2}+r-3)k$ (as projections in $M_{(l_{1}l_{2}+r)k}(C[0,1]))$ and  $rank(p_{0})=3l_{1}k$.
\end{theorem}

\begin{proof} We will apply \cite[Theorem 2.7]{Li3} (using map $a_{i}$ to replace map $b\circ a_{i}$ as in \cite[Remark 2.8]{Li3}) and its proof (see 2.9-2.16 of \cite{Li3}) for the case $Y$ in \cite[Theorem 2.7]{Li3} being [0,1]. As a matter of fact, in the proof of \cite[Theorem 2.7]{Li3}, Li does use that $X$ to be graph, for only one property that any homomorphism from $C(X)$ to $M_{n}C(Y)$ ($Y$ graph) can be approximated arbitrarily well by homomorphisms with distinct spectra. 
By Remark \ref{5.29},  \cite[Theorem 2.7]{Li3}
holds for the case $X\neq\{pt\}$ being any connected simplicial complex and $Y$, a graph.

For finite set $G\subset C(X)$,  and $\varepsilon>0$, choose $\eta>0$ such that $dist(x_{1},x_{2})\leq\eta$
implies $|g(x_{1})-g(x_{2})|<\frac{\varepsilon}{4}$ for all  $g\in G$, as in \cite[2.16]{Li3}.  Without lose of generality, we can assume that the $Sp\phi|_t$ is distinct for any $t\in (0,1)$ and $Sp\phi|_0$ and $Sp\phi|_1$ have  multiplicities exact $k$ as in Corollary \ref{5.27} and Remark \ref{5.28} above. When we go through Li's proof in \cite{Li3}, we need 
to make  the projections $p_{i}$ to satisfy the extra condition:
$$p_{i}(0),p_{i}(1)\in(M_{l_{1}l_{2}+r}(\mathbb{C}))\otimes\textbf{1}_{k}\subseteq M_{(l_{1}l_{2}+r)k}(\mathbb{C}).$$
We will repeat part of the proof of \cite[Theorem 2.7]{Li3} and point out how to modify it.

As in the proof of \cite[Theorem 2.7]{Li3}, we can choose an open cover $U_{0},U_{1},\cdots,U_{\bullet}$ of $[0,1]$
with $$U_{0}=[0,b_{0}),U_{1}=(a_{1},b_{1}),U_{2}=(a_{2},b_{2}),\cdots,U_{\bullet-1}=(a_{\bullet-1},b_{\bullet-1}),U_{\bullet}=(a_{\bullet},1],$$
$$0<a_{1}<b_{0}<a_{2}<b_{1}<a_{3}<b_{2}<\cdots<a_{\bullet}<b_{\bullet-1}<1.$$
~~We will define $P_{U}^{i}(i=1,2,\cdots,l_{1})$ as same as in \cite[2.12]{Li3} for $U=U_{i}(0<i<\bullet)$---note that $Sp\phi_{y}$, for $y\in(a_{1},b_{\bullet-1})\subset(0,1)$, are distinct. For $U_{0}$ and $U_{\bullet}$, a special care is needed as follows. We will only do it for $U_{0}$
(it is the same for $U_{\bullet}$).
Write $Sp\phi|_{0}=\{\lambda_{1}^{\sim k},\lambda_{2}^{\sim k},\cdots,\lambda_{q}^{\sim k}\}$ with $q=l_{1}l_{2}+r$. Then $\{\lambda_{1},\lambda_{2},\cdots,\lambda_{q}\}$ can be paired with $\{a_{0}(0)^{\sim l_{2}},a_{2}(0)^{\sim l_{2}},\cdots,a_{l_{1}-1}(0)^{\sim l_{2}},a_{l_{1}}(0)^{\sim (l_{2}+r)}\}$ (note $\sim\!l_{2}k$ is changed to $\sim\!l_{2}$ here) to within $\eta$. We can divide $\{\lambda_{1},\lambda_{2},\cdots,\lambda_{q}\}$ into groups $\{\lambda_{1},\lambda_{2},\cdots,\lambda_{q}\}=\bigcup_{j=1}^{l_{1}}E'^{j}$ (where $|E'^{j}|=l_{2}$
if $1\leq j\leq l_{1}-1$, and $|E'^{j}|=l_{2}+r$ if $j=l_{1}$) such that $dist(\lambda_{i},a_{j}(0))<\eta$, for all $ \lambda_{i}\in E^{j}$.

Let $\sigma'$ satisfy the following conditions: \\
(1) $\sigma'<\min\{dist(\lambda_{i},\lambda_{j}),\ \ i\neq j\}$;\\
(2) $\sigma'<\eta-\max\{dist(\lambda_{i},a_{j}(0)),\ \lambda_{i}\in E^{j}\}.$

We can choose $b_{1}$ $(>b_{0}>a_{1}>0)$ being so small that for any $y\in[0,b_{1}]$, $Sp\phi_{y}$ and $Sp\phi_{0}$ can be paired to within $\frac{\sigma'}{2}$
and $dist(a_{j}(y),a_{j}(0))<\frac{\sigma'}{2}$. Then for each $y\in[0,b_{1}]$, $Sp\phi_{y}$ can be written as a set of
$$\{\lambda_{1}^{1}(y),\lambda_{1}^{2}(y),\cdots,\lambda_{1}^{k}(y),\lambda_{2}^{1}(y),\lambda_{2}^{2}(y),\cdots,\lambda_{2}^{k}(y),\cdots,\lambda_{q}^{1}(y),\cdots,\lambda_{q}^{k}(y)\}$$
with $\lambda_{i}^{j}(0)=\lambda_{i}$. Then let $E^{j}(y)$ be the set $\{\lambda_{i}^{i'}(y)$; $\lambda_{i}\in E'^{j}\}$. In this way we have,  if $\lambda_{i}^{i'}\in E^{j}$, then 
$$dist(\lambda_{i}^{i'}(y),a_{j}(y))<\eta.$$
Let both $P_{U_{0}}^{j}(y)$ and $P_{U_{1}}^{j}(y)$ (defined on $U_{0}=[0,b_{0})$ and $U_{1}=(a_{1},b_{1})$) be the spectral projections corresponding to $E_{j}(y)$. In particular,  $P_{U_{0}}^{j}(0)\in M_{l_{1}l_{2}+r}(\mathbb{C})\otimes\textbf{1}_{k}$. We can define $p_{j}(y)$ as a subprojection of $P_{U}^{j}(y)$
(for $U\ni y$) as in 2.9-2.16 of \cite{Li3} for each $y\in[b_{0},a_{\bullet}]$ but with rank $(p_{j}(y))=(l_{2}-3)k$ (instead of $l_{2}-3$ in \cite{Li3}) for $1\leq j\leq l_{1}-1$ and $rank(p_{l_{1}}(y))=(l_{2}+r-3)k$ (instead of $l_{2}+r-3$ in \cite{Li3}). Also we can choose an arbitrary sub projection $p_j(0)< P_{U_{0}}^{j}(0)\in M_{l_{1}l_{2}+r}(\mathbb{C})\otimes\textbf{1}_{k}$ of form
$p_{j}(0)=p'_{j}(0)\otimes\textbf{1}_{k}\in M_{l_{1}l_{2}+r}(\mathbb{C})\otimes\textbf{1}_{k}$
  with $rank(p_{j}'(0))=l_{2}-3$ for $1\leq j\leq l_{1}-1$, and
$rank(p'_{l_{1}}(0))=l_{2}+r-3.$ Consequently, $$rank(p_{j}(0))=(l_{2}-3)k~~~~~\mbox{ and}~~~~rank(p_{l_{1}}(0))=(l_{2}+r-3)k.$$
~~~~~~Finally, connect $p_{j}(0)$ and $p_{j}(b_{0})$ by $p_{j}(y)$ for $y\in[0,b_{0}]$ inside $P_{U_{0}}^{j}(y)$. As one can see from 2.16 of \cite{Li3}, if the projections $p_{j}(y)$ are subprojections of $P_{U}^{j}(y)$, then all the estimations in that proof hold. After we do similar modifications for $P_{U_{\bullet}}^{j}(y)$ and $p_{j}(y)$ near point 1, we will get $p_{j}(y)\in M_{l_{1}l_{2}+r}(I_{k})$
instead of $M_{(l_{1}l_{2}+r)k}(C[0,1])$. (This method was also used in the proof of \cite[Theorem 3.10]{EGS}.)\\
\end{proof}

The following result is a generalization  of \cite[Proposition 4.42]{G5}.

\begin{theorem}\label{5.31} Let $X$ be a connected finite simplicial complex of dimension at most 2, $\varepsilon>0$ and $F\subset C(X)$, a finite set of generators. Suppose that $\eta\in(0,\varepsilon)$ satisfies that if $dist(x,x')\leq2\eta$, then $\|f(x)-f(x')\|<\frac{\varepsilon}{4}$ for all $f\in F$.

For any $\delta>0$ and  positive integer $J>0$, there exist an integer $L>0$ and a finite set $H\subseteq AffTC(X)(=C_{\mathbb{R}}(X))$ such that the following holds.

If $\phi,\psi:C(X)\rightarrow B=M_{K}(I_{k})$ (or $B=PM_{\bullet}(C(Y))P$) are unital homomorphisms with the properties: \\
(a) $\phi$ has $sdp(\eta/32,\delta)$;\\
(b) $K\geq L$ (or $rank(P)\geq L$);\\
(c) $\|AffT\phi(h)-AffT\psi(h)\|<\frac{\delta}{4}$, for all $h\in H$,\\
then there are three orthogonal projections $Q_{0},Q_{1}, Q_2 \in B$, 
 two homomorphisms $\phi_{1}\in Hom(C(X),Q_{1}BQ_{1})_{1}$ and $\phi_{2}\in Hom(C(X),Q_{2}BQ_{2})_{1}$, and a unitary $u\in B$ such that\\
(1) $\textbf{1}_B=Q_{0}+Q_{1}+Q_2$;\\
(2) $\|\phi(f)-\big(Q_0\phi(f)Q_0+\phi_{1}(f)+\phi_{2}(f)\big)\|<\varepsilon$ ~~and \\ $\|(Adu\circ\psi)(f)-\big(Q_0(Adu\circ\psi)(f)Q_0+\phi_{1}(f)+\phi_{2}(f)\big)\|<\varepsilon$, for all $f\in F$;\\
(3) $\phi_{2}$ factors through $C[0,1]$;\\
(4) $Q_{1}=p_{1}+\cdots+p_{n}$ with $(rank(Q_{0})+2)J<rank(p_{i})$ $(i=1,2,\cdots,n)$, where rank: $K_{0}(B)\rightarrow \mathbb{Z}$ is  the map induced on $K_{0}$ by the evaluation map at 0 or 1.  (which is $rank~ p_{i}(\underline{0})$ for $B=M_K(I_k)$, where $rank~p_{i}(\underline{0})$ is regarded as projections in $M_{K}(\mathbb{C})$ not $M_{K}(M_{k}(\mathbb{C}))$), and $\phi_1$ is defined by
 $$\phi_{1}(f)=\sum\limits_{i=1}\limits^{n}f(x_{i})p_{i},\ \forall f\in C(X),$$ where $p_{1},p_2, \cdots,p_{n}$ are mutually orthogonal projections and
$\{x_{1},x_{2},\cdots,x_{n}\}\subset X$ is an $\varepsilon$-dense subset of $X$.
\end{theorem}

\begin{proof} For the case $B=PM_{\bullet}(C(Y))P$, this is \cite[Proposition 4.42]{G5}. 
The proof for the case $B=M_K(I_k)$ is almost  the same as the proof of  \cite[Proposition 4.42]{G5}, replacing   \cite[Theorem 4.1]{G5} by
Theorem \ref{5.30} above. The only thing one should notice is that, in \cite[Lemma 4.33]{G5}, $rank\phi(1)=K$, the $K$ should be corresponding to $K$ in our theorem (not $Kk$) and $\Theta(y)$ should be defined as $$\Theta(y)=\left \{\alpha\circ\beta_{1}(y)^{\sim L_{2}k},\alpha\circ\beta_{2}(y)^{\sim L_{2}k},\cdots,\alpha\circ\beta_{L-1}(y)^{\sim L_{2}k},\alpha\circ\beta_{L}(y)^{\sim (L_{2}+L_{1})k}\right \}.$$
(Note in the above, we use $\sim\!\!L_{2}k$ and $\sim\!\!(L_{2}+L_{1})k$ to replace $\sim\!\!L_{2}$ and $\sim\!\!(L_{2}+L_{1})$ in \cite{G5}.) In the proof of this version  of  \cite[Lemma 4.33]{G5}, one can choose the homomorphism $\psi':C(X)\rightarrow M_{k}(C[0,1])$ (not to  $M_{Kk}(C[0,1])$) as the map $\psi$ there, with
$$\|AffT\phi(f)-AffT\psi'(f)\|<\frac{\delta}{4}~~~~~~ \forall f\in H(\eta,\delta,x)$$
as in \cite[Lemma 4.33]{G5}. Then let $\psi=\psi'\otimes\imath_{k}$, where $\imath_{k}:\mathbb{C}\rightarrow M_{k}(\mathbb{C})$ is defined by $\imath_{k}(\lambda)=\lambda\cdot\textbf{1}_{k}$. With this  modification, we have $Sp\psi_{y}'$ being
$$\Theta'(y)=\{\alpha\circ\beta_{1}(y)^{\sim L_{2}},\alpha\circ\beta_{2}(y)^{\sim L_{2}},\cdots,\alpha\circ\beta_{L-1}(y)^{\sim L_{2}},\alpha\circ\beta_{L}(y)^{\sim (L_{2}+L_{1}})\}$$
and $Sp\psi_{y}$ being
$$\Theta(y)=\{\alpha\circ\beta_{1}(y)^{\sim L_{2}k},\alpha\circ\beta_{2}(y)^{\sim L_{2}k},\cdots,\alpha\circ\beta_{L-1}(y)^{\sim L_{2}k},\alpha\circ\beta_{L}(y)^{\sim (L_{2}+L_{1})k}\}$$
as desired. All other parts of the proof are exactly the same.\\
\end{proof}

For the proof of uinqueness theorem in \cite{GJL}, it is important to have a  simultaneous decomposition for two homomorphisms as below.

\begin{theorem}\label{5.32} Let $X$ be a connected finite simplicial complex of dimension at most 2, $\varepsilon>0$ and $F\subset C(X)$, a finite set of generators. Suppose that $\eta\in(0,\varepsilon)$ satisfies that if $dist(x,x')\leq2\eta$, then $\|f(x)-f(x')\|<\frac{\varepsilon}{4}$ for all $f\in F$. Let $\kappa$ be a fixed simplicial structure of $X$.

For any $\delta>0$ and  positive integer $J>0$, there exist an integer $L>0$ and a finite set $H\subseteq AffTC(X)(=C_{\mathbb{R}}(X))$ such that the following holds.

If $X_1$ is a connected sub-complex of $(X, \kappa)$, and  if $\phi,\psi:C(X_1)\rightarrow B=M_{K}(I_{k})$ (or $B=PM_{\bullet}(C(Y))P$) are unital homomorphisms with the following properties: \\
(a) $\phi$ has $sdp(\eta/32,\delta)$;\\
(b) $K\geq L$ (or $rank(P)\geq L$);\\
(c) $\|AffT\phi(h|_{X_1})-AffT\psi(h|_{X_1})\|<\frac{\delta}{4}$, for all $h\in H$,\\
then there are three orthogonal projections $Q_{0},Q_{1}, Q_2 \in B$, two homomorphisms $\phi_{1}\in Hom(C(X_1),Q_{1}BQ_{1})_{1}$ and $\phi_{2}\in Hom(C(X_1),Q_{2}BQ_{2})_{1}$, and a unitary $u\in B$ such that\\
(1) $\textbf{1}_B=Q_{0}+Q_{1}+Q_2$;\\
(2) $\|\phi(f|_{X_1})-\big(Q_0\phi(f|_{X_1})Q_0+\phi_{1}(f|_{X_1})+\phi_{2}(f|_{X_1})\big)\|<\varepsilon$ ~~and \\ $\|(Adu\circ\psi)(f|_{X_1})-\big(Q_0(Adu\circ\psi)(f|_{X_1})Q_0+\phi_{1}(f|_{X_1})+\phi_{2}(f|_{X_1})\big)\|<\varepsilon$ for all $f\in F$;\\
(3) $\phi_{2}$ factors through $C[0,1]$;\\
(4) $Q_{1}=p_{1}+\cdots+p_{n}$ with $(rank(Q_{0})+2)J<rank(p_{i})$ $(i=1,2,\cdots,n)$,
and $\phi_1$ is defined by
$$\phi_{1}(f)=\sum\limits_{i=1}\limits^{n}f(x_{i})p_{i}~~~ \forall f\in C(X),$$ 
where $p_{1},p_2, \cdots,p_{n}$ are mutually orthogonal projections and
$\{x_{1},x_{2},\cdots,x_{n}\}\subset X_1$ is an $\varepsilon$-dense subset of $X_1$.
\end{theorem}

\begin{proof} Suppose that $\{X_i\}_i$ are all connected sub-complexes of $(X,\kappa)$ (there are finitely many of them for a fixed simplicial structure of a finite complex). Apply Theorem \ref{5.30} to each $X_i$ to obtain $L_i$ and $H_i\subseteq AffT(C(X_i))$ as in the theorem. By Tietze  Extension Theorem, there are finite sets ${\tilde H}_i \subseteq AffT(C(X))$ such that $H_i\subseteq \{h|_{X_i}~|~~h\in {\tilde H}_i\} $. Evidently $L=\max_i\{L_i\}$ and $H=\cup_i {\tilde H}_i $ are as desired.\\
\end{proof}

\vspace{3mm}

{\bf Acknowledgement}\ \ The authors would like to express our special thanks of gratitude to Professor Guihua Gong who suggested us to do this interesting problem. 
We also benefit a lot from  discussions with him.


\clearpage

\begin{tiny}

\begin{small}

\end{small}

\end{tiny}

\end{document}